\documentclass[10pt, a4paper, reqno]{amsart}
\usepackage{preambolo}
\usepackage{array}
\usepackage{longtable}
\newcolumntype{C}{>{$}c<{$}}
\usepackage[export]{adjustbox}

\begin{document}


\setcounter{secnumdepth}{3}

\setcounter{tocdepth}{2}

\title[Rank of the Nijenhuis tensor on parallelizable almost complex manifolds]{Rank of the Nijenhuis tensor on parallelizable almost complex manifolds
}

\author[Lorenzo Sillari]{Lorenzo Sillari}

\address{LS: Scuola Internazionale Superiore di Studi Avanzati (SISSA), Via Bonomea 265, 34136, Trieste, Italy.} \email{lsillari@sissa.it}

\author[Adriano Tomassini]{Adriano Tomassini}

\address{AT: Dipartimento di Scienze Matematiche, Fisiche e Informatiche, Unità di Matematica e
Informatica, Università degli Studi di Parma, Parco Area delle Scienze 53/A, 43124, Parma, Italy}
\email{adriano.tomassini@unipr.it}

\maketitle

\begin{abstract} 
\noindent \textsc{Abstract}. We study almost complex structures on parallelizable manifolds via the rank of their Nijenhuis tensor. First, we show how the computations of such rank can be reduced to finding smooth functions on the underlying manifold solving a system of first order PDEs. On specific manifolds, we find an explicit solution. Then we compute the Nijenhuis tensor on curves of almost complex structures, showing that there is no constraint (except for lower semi-continuity) to the possible jumps of its rank. Finally, we focus on $6$-nilmanifolds and the associated Lie algebras. We classify which $6$-dimensional, nilpotent, real Lie algebras admit almost complex structures whose Nijenhuis tensor has a given rank, deducing the corresponding classification for left-invariant structures on $6$-nilmanifolds. We also find a topological upper-bound for the rank of the Nijenhuis tensor for left-invariant almost complex structures on solvmanifolds of any dimension, obtained as a quotient of a completely solvable Lie group. Our results are complemented by a large number of examples.
\end{abstract}

\blfootnote{  \hspace{-0.55cm} 
{\scriptsize 2020 \textit{Mathematics Subject Classification}. Primary: 32Q60, 53C15; Secondary: 53C30, 58H15. }}

\section{Introduction}

Let $(M,J)$ be an almost complex $2m$-manifold; the research of almost complex invariants is a natural problem that arises when studying analytic and geometric properties of $J$. In particular, when $(M,J)$ is a complex manifold, the Dolbeault, Bott-Chern and Aeppli cohomologies, and the Kodaira dimension provide significant invariants of the complex manifold. Very recently those invariants have been introduced for an arbitrary almost complex manifold. In more detail, the Dolbeault cohomology \cite{CW21} and the associated Bott-Chern and Aeppli cohomologies \cite{CPS22} of an almost complex structure, the Kodaira dimension of an almost complex manifold (\cite{CZ20a, CZ20b}, see also \cite{CNT20, CNT21, CNT22}) and more almost complex, almost-Hermitian or almost-K\"ahler invariants \cite{ST21, HPT22, TT22} all contribute to enrich almost complex geometry. When $J$ is not integrable, its Nijenhuis tensor $N_J$ remains one of the most known invariants of $J$, and one of the first that has been introduced \cite{Nij51}. A celebrated theorem of Newlander and Nirenberg \cite{NN57} establishes that $J$ is integrable if and only if $N_J$ identically vanishes.

In addition to characterizing integrability, $N_J$ encodes further information on $J$. Indeed, a possible way to measure how far is $J$ from being integrable is studying the rank of the map
\[
N_J \colon TM \otimes TM \longrightarrow TM,
\]
or, equivalently, the rank of the map
\[
\bar \mu \colon A^{1,0}_x \longrightarrow A^{0,2}_x, \quad x \in M,
\]
defined on $(1,0)$-forms as $\bar \mu := \pi^{0,2} \circ d$. For shortness, we refer to the rank of such map as the \emph{rank of $J$}. The situation farthest from integrability is that of \emph{maximally non-integrable almost complex structures}, i.e., structures for which the rank is maximal at every point. This natural class of almost complex structures, already present in the literature under the name of \emph{totally non-integrable almost complex structures} \cite{dB94, MT97}, play a significant role in almost complex geometry, providing at the same time a large source of examples \cite{Ver08, CW21, CGGH21, CGG22, CGHR22, CPS22}. Coelho, Placini and Stelzig \cite{CPS22} extensively studied maximally non-integrable almost complex structures, proving that they satisfy an $h$-principle. In particular, they show that in dimension $2m \ge 10$ every almost complex structure is homotopic to a maximally non-integrable one \cite[Corollary A.1]{CPS22}. This is true in any dimension if the underlying manifold is parallelizable \cite[Corollary A.4]{CPS22}.

In this paper, we are concerned with the study of almost complex structures of constant rank on parallelizable manifolds, of which integrable and maximally non-integrable structures represent the extremal cases. Our first goal is to explicitly build almost complex structures of arbitrary, prescribed rank (possibly constant) on a parallelizable $2m$-manifold $M$. Starting from an assigned almost complex structure $J_0$ on $M$, we define an almost complex structure $J_1$, depending on $J_0$ and $m^2$ smooth functions. We describe how to compute the map $\bar \mu_1$ associated to $J_1$ in terms of $J_0$, the smooth functions parametrizing $J_1$ and their first order derivatives. Prescribing conditions on the rank of $\bar \mu_1$ is equivalent to finding solutions to a system of first order PDEs. On the Kodaira-Thurston and the Iwasawa manifolds, we are able to explicitly solve the system, producing families of almost complex structures of constant rank parametrized by smooth functions (Propositions \ref{prop:mni:KT} and \ref{prop:families:Iwasawa}).

A consequence of the approach we adopted, is that it is immediate to build small deformations of $J_0$ using $J_1$, obtaining information on the rank of the Nijenhuis tensor in a neighborhood of $J_0$ (Theorem \ref{thm:lower:bound}).

\begin{thmx}\label{intro:lower:bound}
Let $(M,J_0)$ be a parallelizable almost complex manifold. Let $\{ \phi^j \}_{j=1}^m$ be a co-frame of $(1,0)$-forms for $J_0$. Assume that $J_1$ is an almost complex structure defined by a co-frame of $(1,0)$-forms of the form
\[
\omega^j := \phi^j + f^j_k \phi^{\bar k}, \quad j=1,\dots m,
\]
where $f^j_k \in C^\infty(M)$. Consider the curve of almost complex structures $J(s)$, $ s\in \C$, defined by
\[
\omega^j_s := \phi^j + s f^j_k \phi^{\bar k}, \quad \abs{s} < \epsilon, \quad j=1,\dots,m.
\]
Then we have that
\[
\rk N_{J(s)}|_x \ge \max \{\rk N_{J_0}|_x, \rk N_{J_1}|_x  \},
\]
at every point $x \in M$. 
\end{thmx}

Focusing on holomorphically parallelizable complex $3$-solvmanifolds, we further show (Proposition \ref{prop:curves}) that for all $k \in \{ 0,1,2,3 \}$ such manifolds admit curves of almost complex structures $J(s)$, $\abs{s}<\epsilon$, such that $J(0)$ is the standard complex structure and the Nijenhuis tensor of $J(s)$ has constant rank $k$ for $s \neq 0$.

Among parallelizable manifolds, nilmanifolds have been widely employed as a preferred class of examples on which to study (left-invariant) complex structures \cite{CFGU00, Sal01, LU15, COUV16, Yam16}. In the last part of the paper, we focus on $6$-dimensional, nilpotent, real Lie algebras and the associated nilmanifolds. Our main result is a classification of the possible values for the rank of almost complex structures on such Lie algebras, and thus of those of left-invariant almost complex structures on $6$-nilmanifolds. Salamon \cite{Sal01} already established which $6$-dimensional, nilpotent, real Lie algebras admit a complex structure. We take care of the cases where $J$ has rank $k$, for $k =1,2,3$, deducing the possible values for the rank of left-invariant almost complex structures on $6$-nilmanifolds (Theorem \ref{thm:nilmanifolds}).

\begin{thmx}\label{intro:classification}
Let $M= \Gamma \backslash G$ be a $6$-nilmanifold and let $\g$ be the Lie algebra of $G$. Then
\begin{itemize}
    \item [(i)] $M$ admits a left-invariant almost complex structure of rank $3$ if and only if $\g$ is isomorphic to one of
\[
\begin{array}{ll}
(0, 0, 12, 13, 14 + 23, 34 - 25),  & (0, 0, 12, 13, 14, 34 - 25),\\
(0, 0, 12, 13, 14 + 23, 24 + 15),  & (0, 0, 12, 13, 14, 23 + 15),\\
(0, 0, 12, 13, 23, 14),            & (0, 0, 12, 13, 23, 14 - 25),\\
(0, 0, 12, 13, 23, 14 + 25),       & (0, 0, 0, 12, 14 - 23, 15 + 34),\\
(0, 0, 0, 12, 14, 15 + 23),        & (0, 0, 0, 12, 14, 15 + 23 + 24),\\
(0, 0, 0, 12, 14, 15 + 24),        & (0, 0, 0, 12, 13, 14 + 35),\\
(0, 0, 0, 12, 23, 14 + 35),        & (0, 0, 0, 12, 23, 14 - 35),\\
(0, 0, 0, 12, 14, 24),             & (0, 0, 0, 12, 13 - 24, 14 + 23),\\
(0, 0, 0, 12, 14, 13 - 24),        & (0, 0, 0, 12, 13 + 14, 24),\\
(0, 0, 0, 12, 13, 14 + 23),        & (0, 0, 0, 12, 13, 24),\\
(0, 0, 0, 12, 13, 23);             &
\end{array}
\]

\item [(ii)] $M$ does not admit a left-invariant almost complex structure of rank $2$ if and only if $\g$ is isomorphic to one of 
\[
\begin{array}{ll}
(0, 0, 0, 12, 13, 23), & (0, 0, 0, 0, 0, 12 + 34),\\
(0, 0, 0, 0, 0, 12), & (0, 0, 0, 0, 0, 0);
\end{array}
\]
\item [(iii)] $M$ does not admit a left-invariant almost complex structure of rank $1$ if and only if $\g$ is isomorphic to one of
\[
\begin{array}{ll}
(0, 0, 12, 13, 14 + 23, 34 - 25), & (0, 0, 0, 0, 0, 0).
\end{array}
\]
\end{itemize} 
\end{thmx}

In particular, many $6$-nilmanifolds do not admit a left-invariant maximally non-integrable almost complex structure, while all of them admit a non-left-invariant one \cite[Corollary A.4]{CPS22}.

Exploiting our study of almost complex structures on Lie algebras, we establish a topological upper bound on the rank of the Nijenhuis tensor of left-invariant almost complex structures, valid on certain solvmanifolds of any dimension (Theorem \ref{thm:solvmanifold}).

\begin{thmx}\label{intro:solvmanifold}
Let $M = \Gamma \backslash G$ be a solvmanifold. Assume that $G$ is a completely solvable Lie group. Let $J$ be a left-invariant almost complex structure on $M$. Then
\[
\rk N_J \le \dim_\R M - b_1(M).
\]
\end{thmx}

Finally, we complement our classification with a large number of explicit almost complex structures on Lie algebras, of which we provide a basis of $(1,0)$-forms, resulting in a wide database of examples useful when working with almost complex structures on Lie algebras and on nilmanifolds. Combining the examples with our theoretical results, it is possible to produce several curves of almost complex structures on which the rank of the Nijenhuis tensor is known (Proposition \ref{prop:jump}).

\subsection*{Notation and terminology:} we abbreviate the conjugate of a complex form $\overline{ \phi^j}$ to $\phi^{\bar j}$ and the wedge product $\phi^j \wedge \phi^k$ of two forms to $\phi^{jk}$. Given a Lie algebra $\g = \R \langle e_1, \dots, e_k \rangle$, we write it as $k$-tuple $(de^1, \dots, de^k)$, abbreviating $e^{jk}$ to $jk$. We shall write an asterisk $\ast$ instead of $de^j$ to mean any (possibly zero) $2$-form. We denote by $A^k$ both the space of complex $k$-forms on a smooth manifold and the exterior power $\bigwedge^k (\g^*)^\C$, while $A^k_\R$ denotes the corresponding real space. Finally, we will use consistently the Einstein notation, summing over repeated indices.

\subsection*{Acknowledgements:} the authors are partially supported by GNSAGA of INdAM. The second author is partially supported by the Project PRIN 2017 “Real and Complex Manifolds: Topology, Geometry and
holomorphic dynamics”.

\section{Almost complex structures of prescribed rank on parallelizable manifolds}\label{sec:parallelizable}

In this section we illustrate a general procedure that can be used to produce almost complex structures of prescribed rank on a parallelizable $2m$-manifold $M$. Starting from an arbitrary almost complex structure $J_0$, we build a different almost complex structure $J_1$, depending on $J_0$ and $m^2$ smooth functions on $M$. Computing the rank of $\bar \mu_1$ amounts to solving a system of first order PDEs involving the smooth functions parametrizing $J_1$. Even though there is no general method to solve the system, on some specific manifold we are able to do so, producing the desired structures.

\subsection{Outline of the general procedure.}\label{sec:outline}
Let $M$ be a parallelizable $2m$-manifold. Fix a frame of vector fields $\{ E_j \}_{j=1}^{2m}$ giving a parallelism of $M$ and let $J_0$ be an almost complex structure on $M$. The choice of $J_0$ determines a co-frame of $(1,0)$-forms $\{ \phi^j \}_{j=1}^m$. Conversely, taking $m$ independent (over $C^\infty(M)$) complex $1$-forms and declaring them to have type $(1,0)$ completely determines an almost complex structure $J_0$. Assume that the differentials $d \phi^j$ are known. Consider on $M$ the almost complex structure $J_1$, defined by the co-frame of $(1,0)$-forms 
\begin{equation}\label{eq:deformed}
    \omega^j := \phi^j + f^j_k \phi^{\bar k}, \quad j=1,\dots m,
\end{equation}
where the $f^j_k$ are complex-valued, smooth functions on $M$. Set 
\begin{equation}\label{matrix:F}
\Phi := \left( f^j_k \right) \in M_{m \times m} \left( C^\infty (M) \right)
\end{equation}
and let $P$ be the matrix
\begin{equation}\label{matrix:P}
P :=
\begin{bmatrix}
\Id_m & \Phi \\
 & \\
\bar \Phi & \Id_m
\end{bmatrix}.
\end{equation}
The forms $\omega^j$ are independent, and thus $J_1$ is well-defined, as soon as
\[
D := \det(P) \in C^\infty(M)
\]
is a never-vanishing function on $M$. We compute the rank of the map $\bar \mu_1$ associated to $J_1$ on $(1,0)$-forms. The differential of $\omega^j$ is
\[
d \omega^j = d \phi^j + f^j_k d\phi^{\bar k} + df^j_k \wedge \phi^{\bar k}.
\]
We have to express it in function of $\{ \omega^j, \omega^{\bar j} \}_{j=1}^m$ and then take its $(0,2)$-degree part \emph{with respect to the bigrading induced by $J_1$}.

Let $\{\psi_j, \psi_{\bar j} \}$ be a frame of vector fields dual to $\{\omega^j, \omega^{\bar j} \}$. In such a frame, we can write for any $\theta \in C^\infty(M)$
\[
\left( d\theta \wedge \phi^{\bar k} \right)^{0,2} = F^k_{lp} (\theta) \, \omega^{\bar l \bar p}, 
\]
where $F^k_{lp}$ are suitable $(0,1)$-vector fields belonging to $C^\infty(M) \langle \psi_{\bar j} \rangle_{j=1}^m$.

By assumption, we are given an explicit expression for $d\phi^j$, $d\phi^{\bar j}$ that can be written in terms of smooth functions on $M$ and the $2$-forms $\phi^{jk}$, $\phi^{j \bar k}$, $\phi^{\bar j \bar k}$. One can compute the co-frame $\{\phi^j, \phi^{\bar j} \}$ in function of $\{\omega^j, \omega^{\bar j} \}$ by inverting the matrix $P$. Taking the projection on degree $(0,2)$, we obtain an explicit expression for $\bar \mu_1 \omega^j$ in terms of the functions $f^j_k$, of their first order derivatives $F^k_{lp}(f^j_k)$ and of a basis of $(0,2)$-forms $\{ \omega^{\bar j \bar k} \}$. The rank of $\bar \mu_1$ can be prescribed imposing conditions on the $f^j_k, F^k_{lp}(f^j_k)$ and finding functions satisfying such constraints provides a structure with the desired rank.
\vspace{.2cm}

In practice, the approach we described is strongly limited by the difficulty of the computations involved, both on the side of the linear algebra, and on that of solving the final system of PDEs. In our applications, we will put ourselves in the best case scenario, making assumptions based on the following remarks:
\begin{enumerate}
    \item the almost complex structures that can be obtained in this way depend on the initial $J_0$. A simple choice of $J_0$ (e.g., $J_0 E_k = E_{m+k}$, $J_0 E_{m+k} = - E_k$, $k=1,\dots,m$) will allow to immediately find a co-frame of $(1,0)$-forms;
    
    \item increasing the dimension of $M$ drastically increases the difficulty of the computations. We will focus on manifolds of dimension $4$ and $6$;
    
    \item by choosing a manifold for which smooth functions can be explicitly written, or at least put in a manageable form, the final system of PDEs can be solved more easily.
\end{enumerate}

\begin{remark}
In the rest of the paper, we will present several applications of the procedure discussed in this section. However, we do not aim to cover all the possible situations in which our approach can be useful. For instance, it can be used to produce complex structures, similarly at what we do in section \ref{sec:KT} for the Kodaira-Thurston manifold.
\end{remark}

\subsection{The Kodaira-Thurston manifold.}\label{sec:KT}

We use the Kodaira-Thurston manifold as a toy model for the computations of the rank of $\bar \mu$. In dimension $4$, the only possible values for the rank are $0$ and $1$. Due to these restrictions, the low-dimensional examples are less significant. On the other side, they allow to greatly simplify computations and to illustrate clearly the ideas involved.
\vspace{.3cm}

We begin by briefly recalling the construction of the Kodaira-Thurston manifold. Consider the $3$-dimensional Heisenberg group
\[
\mathbb{H}_3 := \left \{ 
\begin{bmatrix}
1 & x & z \\
  & 1 & y \\
  &   & 1 \\
\end{bmatrix},
\quad x,y,z \in \R \right \}.
\]
The Kodaira-Thurston manifold $\KT$ is the $4$-dimensional nilmanifold defined by
\[
\KT:= \mathbb{H}_3 / ( \mathbb{H}_3 \cap SL (3, \Z) ) \times \S^1.
\]
Denoting by $t$ the coordinate on $\S^1$, a basis of left-invariant vector fields on $\KT$ is given by
\[
\left \{ e_1 = \partial_t, \, e_2 = \partial_x, \, e_3 = \partial_y + x \, \partial_z, \, e_4 = \partial_z \right \},
\]
and the dual basis of left-invariant $1$-forms is 
\[
\left \{ e^1 = dt, \, e^2 = dx , \, e^3 = dy, \, e^4 = dz - x \, dy \right \}.
\]
The only non-vanishing Lie bracket on vector fields is $[ e_2, e_3]=e_4$, thus the only non-vanishing differential is $de^4 = - e^{23}$. It is well-known that $\KT$ admits both complex and symplectic structures, but has no K\"ahler structure. It also admits non-integrable almost complex structures.
\vspace{.1cm}

\subsection*{From complex to almost complex.}\label{sec:KT:complex} We start with a fixed complex structure on $\KT$ and we use it to build complex structures and maximally non-integrable structures. 
\vspace{.3cm}

Consider the complex structure on $\KT$ given by
\begin{equation*}
J_0 \, e_1 = - e_4, \qquad J_0 \, e_2 = e_3.
\end{equation*}
A basis of $(1,0)$-forms for $J_0$ is
\[
\phi^1 := dx + i \, dy , \qquad \phi^2 := dz - x \, dy + i \, dt,
\]
and the differentials are
\[
d \phi^1 =0, \qquad d \phi^2 = - \frac{i}{2}\, \phi^{1 \bar 1 }.
\]
Following the outline we gave in \cref{sec:outline}, we consider the co-frame of $(1,0)$-forms
\begin{align*}
\omega^1 &:= \phi^1 + e \, \phi^{ \bar 1} + f \, \phi^{\bar 2}, \\
\omega^2 &:= \phi^2 + g \, \phi^{ \bar 1} + h \, \phi^{\bar 2},
\end{align*}
where $e$, $f$, $g$, $h \in C^\infty (\KT)$ are complex valued, smooth functions on $\KT$. Declaring $\omega^1$, $\omega^2$ to have type $(1,0)$ defines an almost complex structure $J_1$ on $\KT$ as long as
{\small
\[
D:= \det(P) = 1 - \bar f g - f \bar g + \abs{f}^2 \abs{g}^2 - \abs{e}^2 - \abs{h}^2 + \abs{e}^2\abs{h}^2 - f g \bar e \bar h - \bar f \bar g  e h \in C^\infty (\KT),
\]
}

never vanishes on $\KT$, where $P$ is as in \eqref{matrix:P}. We proceed to characterize integrability of $J_1$ in terms of conditions on the functions $e$, $f$, $g$, $h$ and their derivatives.

Direct computations show that 
{
\[
P^{-1} = \frac{1}{D} \,
\begin{bmatrix}
1 - \bar f g - \abs{h}^2 & \bar f e + f \bar h & e \abs{h}^2 - e - fg \bar h & g \abs{f}^2 - f - \bar f e h \\[5pt]
\bar e g + \bar g h & 1-f \bar g - \abs{e}^2 & f \abs{g}^2 -g - \bar g e h & h \abs{e}^2 - h -f g \bar e \\[5pt]
\bar e \abs{h}^2 - \bar e - \bar f \bar g h & \bar g \abs{f}^2 - \bar f - f \bar e \bar h & 1 - f \bar g - \abs{h}^2 & f \bar e + \bar f h \\[5pt]
\bar f \abs{g}^2 -\bar g - g \bar e \bar h & \bar h \abs{e}^2 - \bar h -\bar f \bar g e & e \bar g + g \bar h & 1- \bar f g - \abs{e}^2\\
\end{bmatrix}.
\]
}

We express the $\phi^j$ in function of the $\omega^j$, obtaining
\begin{align*}
\phi^1 = \frac{1}{D} \Big[ &(1 - \bar f g - \abs{h}^2) \, \omega^1 + (\bar f e + f \bar h) \, \omega^2 \\
+ &(e \abs{h}^2 - e - fg \bar h) \, \omega ^ {\bar 1} + (g \abs{f}^2 - f - \bar f e h) \, \omega ^{\bar 2} \Big], \\
\phi^2 = \frac{1}{D} \Big[ &(\bar e g + \bar g h) \, \omega^1 + (1-f \bar g - \abs{e}^2) \, \omega^2 \\
+ &(f \abs{g}^2 -g - \bar g e h) \, \omega^{\bar 1} + (h \abs{e}^2 - h -f g \bar e) \, \omega^{\bar 2} \Big].
\end{align*}

The differentials of the $\omega^j$ are
\begin{align}\label{eq:diff:omega}
\begin{split}
d \omega^1 &= f \, d \phi^{\bar 2} + de \wedge \phi^{\bar 1} + df \wedge \phi^{\bar 2}, \\
d \omega^2 &= (1+h) \, d \phi^{\bar 2} + dg \wedge \phi^{\bar 1} + dh \wedge \phi^{\bar 2}.
\end{split}
\end{align}
Since
\[
\phi^{1 \bar 1} = \frac{1}{D} \left[ - \bar f \, \omega^{1 2} - (1-\abs{h}^2) \, \omega^{1 \bar 1} + \bar f h \, \omega^{1 \bar 2} - f \bar h \, \omega^{\bar 1 2} - \abs{f}^2 \, \omega^{2 \bar 2} + f \, \omega^{\bar 1 \bar 2} \right],
\]
we have that
\begin{equation}\label{eq:diff:basis}
(d \phi^{\bar 2})^{0,2} = - \frac{i}{2} \, (\phi^{1 \bar 1})^{0,2} =  - \frac{i}{2} \frac{f}{D} \omega^{\bar 1 \bar 2}.
\end{equation}

Denote by $\{\xi_j, \, \xi_{\bar j} \}$ the dual frame to $\{\phi^j, \, \phi^{\bar j} \}$ and by $\{\psi_j, \, \psi_{\bar j} \}$ the dual frame to $\{\omega^j, \, \omega^{\bar j} \}$. We can write the $(0,2)$-degree part of the $2$-forms $d \theta \wedge \phi^{\bar j}$ as
\begin{equation}\label{eq:diff:function}
(d \theta \wedge \phi^{\bar j})^{0,2} = F^j (\theta) \, \omega ^{\bar 1 \bar 2},
\end{equation}
where
\begin{align}\label{der:KT}
\begin{split}
    F^1(\theta) &:= \frac{1}{D} \left[ (f \bar e + \bar f h) \, \psi_{\bar 1}(\theta) - (1 - f \bar g - \abs{h}^2) \, \psi_{\bar 2}(\theta)   \right], \\[3pt]
    F^2(\theta) &:= \frac{1}{D} \left[ (1-\bar f g - \abs{e}^2) \, \psi_{\bar 1}(\theta) - (e \bar g + \bar g h) \, \psi_{\bar 2}(\theta)   \right].
    \end{split}
\end{align}
By duality, the frame $\{ \psi_j, \psi_{\bar j} \}$ depends on $\{ \xi_j, \xi_{\bar j} \}$ via $(P^{-1})^T$.

Finally we can compute $\bar \mu_1 \omega^j$ taking the $(0,2)$-degree part of \eqref{eq:diff:omega}. Using \eqref{eq:diff:basis} and \eqref{eq:diff:function} we obtain
\begin{align*}
    \bar \mu_1 \omega^1 &= \left( -\frac{i}{2D}f^2 + F^1(e) + F^2 (f) \right) \omega^{\bar 1 \bar 2}, \\
    \bar \mu_1 \omega^2 &= \left( -\frac{i}{2D}f(1+h) + F^1(g) + F^2 (h) \right) \omega^{\bar 1 \bar 2}.
\end{align*}
The following system of PDEs
\[
\begin{cases}
-\frac{i}{2D}f^2 + F^1(e) + F^2 (f)=0, \\[3pt]
-\frac{i}{2D}f(1+h) + F^1(g) + F^2 (h) =0,
\end{cases}
\]
allows to compute whether or not $J_1$ is integrable. First, we look for constant solutions, that produce left-invariant almost complex structures on $\KT$. The system reduces to
\[
\begin{cases}
-\frac{i}{2}f^2=0, \\[3pt]
-\frac{i}{2}f(1+h)=0.
\end{cases}
\]
We can conclude that the structure $J_1$ is integrable if and only if $f=0$. In any other case, the rank of the left-invariant structure is equal to $1$. This is true as long as $D$ does not vanish, i.e., as long as $\abs{e} \neq 1$ and $\abs{h} \neq 1$.

We now aim at finding maximally non-integrable almost complex structures that are not left-invariant, looking for functions such that at least one between $\bar \mu_1 \omega^1$ and $\bar \mu_1 \omega^2$ never vanishes. To simplify the computations, we take $e$ and $h$ to be identically $0$. We must find $f$, $g \in C^\infty(\mathcal{KT})$ such that at every point
\[
-\frac{i}{2D}f^2 + F^2 (f) \neq 0  \quad \text{or} \quad
-\frac{i}{2D}f + F^1(g) \neq 0.
\]
The terms involved have the expression
\begin{align*}
    \psi_1 &= \frac{1}{D} \left[ (1 - \bar f g) \, \xi_1 + \bar g ( \bar f g -1) \, \xi_{\bar 2} \right ], \\[3pt]
    \psi_2 &= \frac{1}{D} \left[ (1 - f \bar g) \, \xi_2 + \bar f ( f \bar g -1) \, \xi_{\bar 1} \right],
\end{align*}
where $D=(1-f \bar g)(1- \bar f g)$ must be never-vanishing. We keep computing
\begin{align*}
    F^1(g) &= - \frac{1}{D} (1 - f \bar g) \, \psi_{\bar 2}(g) = -\frac{1}{D} \left(\xi_{\bar 2} - f \xi_1 \right) (g), \\[3pt]
    F^2(f) &= \frac{1}{D} (1-\bar f g) \, \psi_{\bar 1}(f)= \frac{1}{D} \left(\xi_{\bar 1} - g \xi_2 \right) (f),
\end{align*}
obtaining 
\[
-\frac{i}{2D}f^2 + \frac{1}{D} \left(\xi_{\bar 1} - g \xi_2 \right) (f) \neq 0 \quad \text{or} \quad
-\frac{i}{2D}f - \frac{1}{D} \left(\xi_{\bar 2} - f \xi_1 \right) (g) \neq 0.
\]

\begin{proposition}\label{prop:mni:KT}
For every never vanishing $f \in C^\infty(\KT)$, there exists a maximally non-integrable almost complex structure $J_f$ on $\KT$.
\end{proposition}

\begin{proof}
We impose $g=0$, so that the equations reduce to
\[
-\frac{i}{2D}f^2 + \frac{1}{D} \xi_{\bar 1}(f) \neq 0 \quad \text{or} \quad 
-\frac{i}{2D}f \neq 0,
\]
and any never vanishing $f$ provides a maximally non-integrable almost complex structure $\omega^1 = \phi^1 + f \, \phi^{\bar 2}$, $\omega^2 = \phi^2$. Note that if $e=g=h=0$ then $D=1$, and the resulting structure is well defined.
\end{proof}

\subsection*{From almost complex to complex.}\label{sec:KT:mni}
We apply again our machinery, starting from a non-integrable $J_0$ and recovering an integrable structure $J_1$. Consider the structure of constant rank $1$ defined by
\[
J_0 \, e_1 = - e_2, \qquad J_0 \, e_3 = - e_4.
\]
A basis of $(1,0)$-forms is
\[
\phi^1 := dx + i \, dt , \qquad \phi^2 := dz - x \, dy + i \, dy,
\]
and the differentials are
\[
d \phi^1 =0, \qquad d \phi^2 = \frac{i}{4}\, \left( \phi^{1 2} - \phi^{1 \bar 2} + \phi^{\bar 1 2} - \phi^{\bar 1 \bar 2} \right).
\]
Proceeding in the same way as in the previous paragraph, we build a structure $J_1$ depending on functions $e$, $f$, $g$, $h \in C^\infty(\KT)$. The $2$-forms $\phi^{12}$, $\phi^{1 \bar 2}$, expressed in the co-frame $\{ \omega^j, \omega^{\bar j} \}$ are
\begin{align*}
\phi^{1 2} &= \frac{1}{D} \left[ \omega^{1 2} - g \, \omega^{1 \bar 1} - h \, \omega^{1 \bar 2} - e \, \omega^{\bar 1 2} + f \, \omega^{2 \bar 2} + (eh - fg) \, \omega^{\bar 1 \bar 2} \right], \\
\phi^{1 \bar 2} &= \frac{1}{D} \left[ -\bar h \, \omega^{1 2} + g \bar h \, \omega^{1 \bar 1} + (1 - \bar f g) \, \omega^{1 \bar 2} + e \bar h \, \omega^{\bar 1 2} + \bar f e \, \omega^{2 \bar 2} -e \, \omega^{\bar 1 \bar 2} \right],
\end{align*}
and the only non-zero differential is
\begin{align*}
d \phi^2 = \frac{i}{4D} \Big[ &(1 + \bar h - \bar e -\bar e \bar h + \bar f \bar g) \, \omega^{1 2} + (- g \bar h - \bar g h -g - \bar g) \, \omega^{1 \bar 1} \\ 
+ &(-1 + \bar f g + \bar e h -h + \bar e) \, \omega^{1 \bar 2} + (1 - f \bar g - e \bar h - e + \bar h) \, \omega^{\bar 1 2} \\
+ &(-\bar f e - f \bar e + f + \bar f) \, \omega^{2 \bar 2} + (-1 +e -h +e h - fg) \, \omega^{\bar 1 \bar 2} \Big].
\end{align*}
The corresponding $\bar \mu_1$ on $(1,0)$-forms is
\begin{align*}
    \bar \mu_1 \omega^1 &= \left( -\frac{i}{4D}(1-e+h-eh+fg)f + F^1(e) + F^2 (f) \right) \omega^{\bar 1 \bar 2}, \\
    \bar \mu_1 \omega^1 &= \left( -\frac{i}{4D}(1-e+h-eh+fg)(1+h) + F^1(g) + F^2 (h)  \right) \omega^{\bar 1 \bar 2},
\end{align*}
where $F^j$ are as in equation \eqref{der:KT}. In order to find an integrable structure we must solve the following system of PDEs
{\small
\begin{equation}\label{system:KT}
\begin{cases}
-\frac{i}{4}(1-e+h-eh+fg)f + (f \bar e + \bar f h) \, \psi_{\bar 1} (e) - (1 - f \bar g - \abs{h}^2) \, \psi_{\bar 2} (e) \\[3pt]
+ (1 - \bar f g - \abs{e}^2) \, \psi_{\bar 1} (f) - (\bar g e + g \bar h) \, \psi_{\bar 2} (f) =0, \\[3pt]
-\frac{i}{4}(1-e+h-eh+fg)(1+h) + (f \bar e + \bar f h) \, \psi_{\bar 1} (g) - (1 - f \bar g - \abs{h}^2) \, \psi_{\bar 2} (g) \\[3pt]
+ (1 - \bar f g - \abs{e}^2) \, \psi_{\bar 1} (h) - (\bar g e + g \bar h) \, \psi_{\bar 2} (h) =0, \\
\end{cases}
\end{equation}
}
where the operators $\psi_{\bar j}$ have the expression
\begin{align*}
    \psi_{\bar 1} = \frac{1}{D} \Big[ &(\bar e \abs{h}^2 - \bar e - \bar f \bar g h) \, \xi_1 +  (\bar g \abs{f}^2 - \bar f - f \bar e \bar h) \, \xi_2 \\[5pt]
    + &(1 - f \bar g - \abs{h}^2) \, \xi_{\bar 1} +  (f \bar e + \bar f h) \, \xi_{\bar 2} \Big], \\[5pt]
    \psi_{\bar 2} = \frac{1}{D} \Big[ &(\bar f \abs{g}^2 -\bar g - g \bar e \bar h) \, \xi_1 + (\bar h \abs{e}^2 - \bar h -\bar f \bar g e) \, \xi_2 \\[5pt]
    + &(e \bar g + g \bar h) \, \xi_{\bar 1} + (1- \bar f g - \abs{e}^2) \, \xi_{\bar 2} \Big].
\end{align*}
Left-invariant complex structures (i.e., constant solutions) are obtained solving
\[
\begin{cases}
-\frac{i}{4}(1-e+h-eh+fg)f =0, \\[3pt]
-\frac{i}{4}(1-e+h-eh+fg)(1+h) =0. 
\end{cases}
\]

The solutions $\{f=0, \, h=-1\}$ or $\{f=0, \, e=1\}$ lead to a vanishing $D$, thus $J_1$ is a left-invariant complex structure if and only if $1-e+h-eh+fg=0$ and $D$ does not vanish. This can be easily achieved by setting $f=e-1 \neq 0$ and $g=h+1 \neq 0$.

Regarding non-constant solutions, finding an expression for the general solutions of \eqref{system:KT} is a difficult task. We rather provide a family of explicit solutions. 

\begin{proposition}\label{prop:rank0:KT}
For every never vanishing $f \in C^\infty(\KT)$, depending only on the $x$ variable, there exists an integrable almost complex structure $J_f$ on $\KT$.
\end{proposition}

\begin{proof}
Recall that the frame $\{\xi_1, \xi_2 \}$ of $(1,0)$-vector fields can be written in terms of real vector fields as
\[
\xi_1 = \frac{1}{2} \left( \partial_x - i \partial_t \right), \quad \xi_2 = \frac{1}{2} \left( \partial_z - i (\partial_y + x \partial_z) \right).
\]
Let $f \in C^\infty(\KT)$ be a never vanishing function that depends only on the $x$ variable (e.g., $f(x) = A + \cos( 2 \pi x)$, where $A$ is a real constant such that $A > 1$ or $A , -1$). It is not hard to check that the following quadruple
\[
e = f+1, \quad f, \quad g = f, \quad h = f-1,
\]
is a solution of \eqref{system:KT}, and that the corresponding $J_1$ is a well-defined complex structure on $\KT$.
\end{proof}

\section{Holomorphically parallelizable complex \texorpdfstring{$3$}{}-solvmanifolds}\label{sec:holomorphic}

We explicitly compute almost complex structures of arbitrary constant rank on complex parallelizable solvmanifolds of complex dimension $3$. There are three such manifolds \cite{Nak75}, namely the \emph{torus}, the \emph{Iwasawa manifold} and the holomorphically parallelizable \emph{Nakamura manifold}. They all admit a maximally non-integrable almost complex structure (\cite{CPS22}, Corollary A.4). However, only the last one admits a left-invariant maximally non-integrable structure.

\subsection{Torus.}\label{sec:torus}

For the case of the torus $T^6$, we refer the reader to \cite[Example 2.5.2]{CPS22}, where it is remarked that on $T^6$ every left-invariant almost complex structure is integrable, and that one can build on $T^6$ almost complex structures of arbitrary constant rank.

\subsection{Iwasawa Manifold.}\label{sec:Iwasawa}
For the Iwasawa manifold, we do not show the full computations, since they follow closely those performed in Section \ref{sec:KT} for the Kodaira-Thurston manifold, with the complication of working in dimension $6$. We find the possible rank for left-invariant structures and fill the gaps exhibiting explicit non-left-invariant structures. 
\vspace{.3cm}

Consider the complex Heisenberg group
\[
\mathbb{H}_3^\C = \left \{ 
\begin{bmatrix}
1 & z_1 & z_3 \\
  & 1 & z_2 \\
  &   & 1 \\
\end{bmatrix}
: z_1,z_2,z_3 \in \C
\right\},
\]
where the group operation is induced by matrix multiplication. The Iwasawa manifold is the quotient
\[
\I:= \mathbb{H}_3^\C/ \left( \mathbb{H}_3^\C \cap SL(3, \Z[i]) \right).
\]
The complex structure inherited from $\C$ gives a basis of $(1,0)$-forms $\{ \phi^j \}_{j=1}^3$ with differentials
\[
d \phi^1 =0, \qquad d \phi^2=0, \qquad d \phi^3 = - \phi^{12},
\]
that corresponds to a complex structure on the Lie algebra 
\[
(0, 0, 0, 0, 13 - 24, 14 + 23).
\]

Consider the $1$-forms
\begin{align*}
\omega^1 &:= \phi^1 + e \, \phi^{ \bar 1} + f \, \phi^{\bar 2} + g \, \phi^{\bar 3}, \\
\omega^2 &:= \phi^2 + p \, \phi^{ \bar 1} + q \, \phi^{\bar 2} + r \, \phi^{\bar 3}, \\
\omega^3 &:= \phi^3 + s \, \phi^{ \bar 1} + t \, \phi^{\bar 2} + u \, \phi^{\bar 3},
\end{align*}
where $e,f,g,p,q,r,s,t,u \in C^\infty (\mathcal{I})$. Let $P$ be as in \eqref{matrix:P}. As long as $D:= \det (P) \in C^\infty(\I)$ never vanishes, declaring $\omega^j$ to have type $(1,0)$ defines an almost complex structure $J_1$ on $\I$.\\
We first focus on left-invariant structures, that correspond to taking the functions $e,\dots, u$ to be constant. Following \cref{sec:outline}, we are able to build left-invariant structures of rank $1$ and $2$ (see also \cref{sec:invariant} and \cref{table}). A left-invariant structure of rank $2$ is given by
\[
\omega^1 = \phi^1 + \phi^{\bar 3}, \quad  \omega^2 = \phi^2 + 2 \phi^{\bar 2}, \quad \omega^3 = \phi^3,
\]
while one of rank $1$ is given by
\[
\omega^1 = \phi^1, \quad  \omega^2 = \phi^2 + \phi^{\bar 3}, \quad \omega^3 = \phi^3.
\]
Regarding maximally-non-integrable structures, the following Proposition follows from Corollary \ref{cor:mni:nla}, that we prove in Section \ref{sec:invariant}.

\begin{proposition}
The Iwasawa manifold admits no left-invariant maximally-non-integrable almost complex structures.
\end{proposition}

To find a maximally non-integrable almost complex structure on $\I$, let $e,f,g,q,r,u$ be complex smooth functions on $\I$, and consider the following $(1,0)$-forms
\begin{align}\label{eq:def:iwasawa}
\begin{split}
\omega^1 &:= \phi^1 + e \, \phi^{ \bar 1} + f \, \phi^{\bar 2} + g \, \phi^{\bar 3}, \\ 
\omega^2 &:= \phi^2  + q \, \phi^{\bar 2} + r \, \phi^{\bar 3}, \\
\omega^3 &:= \phi^3  + u \, \phi^{\bar 3}. 
\end{split}
\end{align}

The resulting structure $J_1$ is well defined as long as $D=(1-\abs{e}^2) \, (1-\abs{q}^2) \, (1-\abs{u}^2) \neq 0$. Following again \cref{sec:outline}, we write for any $\theta \in C^\infty(\I)$,
\[
(d\theta \wedge \phi^{\bar j})^{0,2} = \sum_{\substack{k, \, l = 1 \\ k < l }}^3 F^{\bar j}_{\bar k \bar l}(\theta) \, \omega^{\bar k \bar l}.
\]
The explicit expressions for the $F^{\bar j}_{\bar k \bar l}$ in terms of the frame $\{ \xi_j, \xi_{\bar j} \}$ dual to $\{ \phi^j, \phi^{\bar j} \}$ are
{\small
\begin{align}
    F^{\bar 1}_{\bar 1 \bar 2} = \frac{1}{D} [ &(1-\abs{u}^2) \, f \, \xi_1 + (1-\abs{u}^2) \, q \, \xi_2 - (1-\abs{u}^2) \, \xi_{\bar 2} ], \label{op:1} \\
    F^{\bar 1}_{\bar 1 \bar 3} = \frac{1}{D} [ &(g - g \abs{q}^2 + fr \bar q + f u \bar r) \, \xi_1 + (r + q u \bar r) \, \xi_2 - u \, \xi_3 - (r \bar q + u \bar r ) \, \xi_{\bar 2} + \xi_{\bar 3} ],  \label{op:2} \\
    F^{\bar 1}_{\bar 2 \bar 3} = \frac{1}{D} [ &(-r \abs{f}^2 + g q \bar f - f u \bar g) \, \xi_1 + (-g q \bar e + f r \bar e - q u \bar g ) \, \xi_2 + u(f \bar e + q \bar f) \, \xi_3  \nonumber \\ 
    &+(g \bar e + r \bar f + u \bar g ) \, \xi_{\bar 2} - (f \bar e + q \bar f) \, \xi_{\bar 3 } ], \label{op:3} \\
    F^{\bar 2}_{\bar 1 \bar 2} = \frac{1}{D} [ &-(1-\abs{u}^2)\, e \, \xi_1 + (1-\abs{u}^2) \, \xi_{\bar 1} ], \label{op:4} \\
    F^{\bar 2}_{\bar 1 \bar 3} = \frac{1}{D} [ &- e(r \bar q + u \bar r ) \, \xi_1 + (r \bar q + u \bar r )  \, \xi_{\bar 1}], \label{op:5}\\
    F^{\bar 2}_{\bar 2 \bar 3} = \frac{1}{D} [ &(g + e r \bar f + e u \bar g) \, \xi_1 + (1-\abs{e}^2) \, r \, \xi_2 + (1-\abs{e}^2)\, \xi_3 \nonumber \\
    & - (g \bar e + r \bar f + u \bar g) \, \xi_{\bar 1} - (1-\abs{e}^2) \, \xi_{\bar 3} ], \label{op:6}\\
    F^{\bar 3}_{\bar 1 \bar 2} = 0,\,\,\,\,&  \label{op:7}\\
    F^{\bar 3}_{\bar 1 \bar 3} = \frac{1}{D} [ &e \, \xi_1 - \xi_{\bar 1} ], \label{op:8} \\
    F^{\bar 3}_{\bar 2 \bar 3} = \frac{1}{D} [ &-(f + e q \bar f) \, \xi_1 - (1-\abs{e}^2) \, q \, \xi_2 + (f \bar e + q \bar f ) \, \xi_{\bar 1} + (1-\abs{e}^2) \, \xi_{\bar 2} ].  \label{op:9}
\end{align}
}

This allows to write $\bar \mu_1$ on $(1,0)$-forms: 
\begin{align*}
    \bar \mu_1 \omega^1 &= \left( -g \, (1-\abs{u}^2) + F^{\bar 1}_{\bar 1 \bar 2} (e) + F^{\bar 2}_{\bar 1 \bar 2}(f) \right) \omega^{\bar 1 \bar 2} \\
    &+ \left( -g \, (r \bar q + u \bar r) + F^{\bar 1}_{\bar 1 \bar 3}(e) + F^{\bar 2}_{\bar 1 \bar 3} (f) + F^{\bar 3}_{\bar 1 \bar 3}(g) \right) \omega^{\bar 1 \bar 3}\\
    &+\left( g \, (g \bar e + r \bar f + u \bar g) + F^{\bar 1}_{\bar 2 \bar 3}(e) + F^{\bar 2}_{\bar 2 \bar 3}(f) + F^{\bar 3}_{\bar 2 \bar 3}(g) \right) \omega^{\bar 2 \bar 3},\\[3pt]
    \bar \mu_1 \omega^2 &= \left( -r \, (1 - \abs{u}^2) + F^{\bar 2}_{\bar 1 \bar 2} (q) \right) \omega^{\bar 1 \bar 2} \\
    &+ \left( -r \, (r \bar q + u \bar r) + F^{\bar 2}_{\bar 1 \bar 3}(q) + F^{\bar 3}_{\bar 1 \bar 3}(r) \right) \omega^{\bar 1 \bar 3}\\
    &+\left( r \, (g \bar e + r \bar f + u \bar g) + F^{\bar 2}_{\bar 2 \bar 3}(q) + F^{\bar 3}_{\bar 2 \bar 3}(r) \right) \omega^{\bar 2 \bar 3},\\[3pt]
    \bar \mu_1 \omega^3 &= \left( - (eq +u ) \, (1 - \abs{u}^2) \right) \omega^{\bar 1 \bar 2} \\
    &+ \left(  - \left( er + r u \bar q + u \,(eq+u) \, \bar r \right) + F^{\bar 3}_{\bar 1 \bar 3}(u) \right) \omega^{\bar 1 \bar 3}\\
    &+\left( \left( gq - fr + g u \bar e + r u \bar f + u \, (eq+u) \, \bar g \right) + F^{\bar 3}_{\bar 2 \bar 3 }(u)  \right) \omega^{\bar 2 \bar 3}.\\
\end{align*}
To further simplify computations, we impose $e=u=0$. Taking the determinant of the coefficients of $\bar \mu_1 \omega^j$, we deduce that $J_1$ is maximally non-integrable if and only if
\[
G = \left( g \, q - f \, r \right) \left[ g \, \xi_{\bar 1}(r) - r \, \xi_{\bar 1}(g) + \frac{1}{D} \left( \xi_{\bar 1} (g) \, \xi_{\bar 1} (q) - \xi_{\bar 1} (f) \, \xi_{\bar 1} (r) \right) \right] \in C^\infty(\I)
\]
never vanishes on $\I$. In terms of the $z_j$ coordinates, we have that $\xi_{\bar 1} = \frac{\partial}{\partial z_{\bar 1}}$. Denote by $x_1$ the real part of $z_1$. The following choice of functions leads to a non-vanishing $G$ at every point:
{\footnotesize
\begin{equation}\label{eq:rank3:iwsawa}
g(x_1) = \cos (2 \pi x_1), \quad r(x_1) = \sin (2 \pi x_1), \quad q(x_1) = \frac{1}{2} \cos (2 \pi x_1), \quad f(x_1) = \frac{1}{2} \sin (2 \pi x_1),
\end{equation}
}

providing a non-left-invariant maximally non-integrable almost complex structure on $\I$. Furthermore, it is immediate to check, using the explicit expression for $\bar \mu_1$, that the functions
\begin{equation}\label{eq:rank1:iwsawa}
f=0, \quad q=0, \quad g = \theta, \quad r = \theta,
\end{equation}
where $\theta \in C^\infty(\I)$ is a never-vanishing function, give a non-left-invariant structure of constant rank $1$ on $\I$, while the choice 
\begin{equation}\label{eq:rank2:iwsawa}
f=0, \quad q=0, \quad g = \sin (2 \pi x_1), \quad r = \cos (2 \pi x_1),
\end{equation}
gives a structure of constant rank $2$. We proceed to build families of non-left-invariant almost complex structures of prescribed constant rank.

\begin{proposition}\label{prop:families:Iwasawa}
For every never vanishing $\theta \in C^\infty(\I)$ such that $\xi_{\bar 1} (\theta) =0$ and $\abs{\theta} \le 1$, there exists a maximally non-integrable almost complex structure $J^3_\theta$ on $\I$.\\
For every never vanishing $\theta \in C^\infty(\I)$ such that $\xi_{\bar 1} (\theta) =0$, there exists an almost complex structure $J^2_\theta$ of constant rank $2$ on $\I$.\\
Furthermore, for every never vanishing $\theta \in C^\infty(\I)$, there exists an almost complex structure $J^1_\theta$ of constant rank $1$ on $\I$.
\end{proposition}

\begin{proof}
Let $\theta \in C^\infty(\I)$ be a never-vanishing function and let $J_\theta$ be the almost complex structure defined by the $(1,0)$-forms
\begin{align*}
\begin{split}
\omega^1 &:= \phi^1 + \theta f \, \phi^{\bar 2} + \theta g \, \phi^{\bar 3}, \\ 
\omega^2 &:= \phi^2  + \theta q \, \phi^{\bar 2} + \theta r \, \phi^{\bar 3}, \\
\omega^3 &:= \phi^3,
\end{split}
\end{align*}

where $f,g,q,r \in C^\infty (\I)$. The family of structures $J^1_\theta$ is obtained choosing the functions $f$, $g$, $q$, $r$ as in \eqref{eq:rank1:iwsawa}, since in this case the rank of $J_\theta$ is $1$. Further assuming that $\xi_{\bar 1}( \theta)=0$, the choice of $f$, $g$, $q$, $r$ as in \eqref{eq:rank2:iwsawa}, provides a family of structures of constant rank $2$. If we also assume $\abs{\theta} \le 1$, and take $f$, $g$, $q$, $r$ as in \eqref{eq:rank3:iwsawa}, then $J_\theta$ is a well-defined maximally-non-integrable almost complex structure.
\end{proof}

\subsection{Nakamura manifold.}\label{subsec:Nakamura}

The Nakamura manifold admits left-invariant structures of any constant rank. Since the computations are substantially the same as in \cref{sec:Iwasawa}, we omit the details. 
\vspace{.3cm}

Let $G$ be the Lie Group $\C \ltimes_\psi \C^2$, with coordinates $z_1$, $z_2$, $z_3$, where
\[
\psi(z_1) =
\begin{bmatrix}
e^{z_1} & 0 \\
0 & e^{-z_1}
\end{bmatrix}.
\]
The Nakamura manifold is the quotient
\[
\Nak := \Gamma \backslash G,
\]
where $\Gamma \subset G$ is a suitable lattice \cite{Nak75}. A basis of holomorphic $(1,0)$-forms that trivializes the complexified tangent bundle is given by
\[
\phi^1 = dz^1, \quad \phi^2 = e^{-z_1} dz^2, \quad \phi^2 = e^{z_1} dz^3,
\]
and their differentials are
\[
d \phi^1 =0, \quad d\phi^2 = - \phi^{12}, \quad d \phi^3 = \phi^{13}.
\]
$\Nak$ admits left-invariant structures of all possible ranks. We give an explicit example for each one. A structure of rank $3$ is defined by the $(1,0)$-forms
\begin{equation}\label{eq:rank3:nakamura}
\omega^1 = \phi^1 + \phi^{\bar 2} + \phi^{\bar 3}, \quad \omega^2 = \phi^2 + \frac{1}{2} \phi^{\bar 2}, \quad \omega^3 = \phi^3 + \frac{1}{2} \phi^{\bar 3}.
\end{equation}
A structure of rank $2$ is given by
\[
\omega^1 = \phi^1, \quad \omega^2 = \phi^2 + \phi^{\bar 3}, \quad \omega^3 = \phi^3 + 2 \phi^{\bar 2},
\]
while a structure of rank $1$ is given by
\[
\omega^1 = \phi^1, \quad \omega^2 = \phi^2 + \phi^{\bar 3}, \quad \omega^3 = \phi^3.
\]


\section{Curves of almost complex structures}\label{sec:curves}

In this section we slightly modify the argument from \cref{sec:parallelizable} to obtain a one-parameter version of our results. This shows that, under certain assumptions, in every neighborhood of a given almost complex structure on a parallelizable manifold there exists a maximally non-integrable structure. We also produce curves of almost complex structures on holomorphically parallelizable complex $3$-solvmanifolds along which the rank of $\bar \mu$ jumps from $0$ to $k \in \{ 1,2,3 \}$. More examples of such curves are presented later in section \ref{sec:homogeneous}.
\vspace{.3cm}

We briefly recall the theory of small deformations of almost complex structures. We refer the reader to the classical literature (e.g., \cite[Chapter~6]{Huy05}) for further details.

Let $M$ be an almost complex manifold and let $J$, $J'$ be two almost complex structures on $M$. The complexified tangent bundle admits two splittings
\[
    TM^\C \cong TM_{1,0} \oplus TM_{0,1} \, \, \text{ and } \, \, TM^\C \cong TM'_{1,0} \oplus TM'_{0,1}, 
\]
induced by $J$, respectively $J'$. If $J$ and $J'$ are close enough in the $C^0$ topology, then the projection
\[
\pi_{0,1} \colon  TM'_{0,1} \longrightarrow  TM_{0,1}
\]
is invertible and the map
\[
\Psi \colon TM_{0,1} \xrightarrow{\, \, \pi_{0,1}^{-1}} \, \, \, TM'_{0,1} \xrightarrow{\, \, \pi_{1,0} \, \,} TM_{1,0},
\]
is well-defined. One can think of $\Psi$ as an element of 
\[
T^*M^{0,1} \otimes TM_{1,0}.
\]
Conversely, fixed an almost complex structure $J$, any $(0,1)$-form with values in $TM_{1,0}$ defines a unique almost complex structure $J'$, by determining its $(1,0)$-forms.

Let now $M$ be a \emph{parallelizable} almost complex manifold, and fix an almost complex structure $J_0$ on $M$. Let $J(s)$ be the curve of almost complex structures on $M$ defined by the $(1,0)$-forms
\[
\omega^j_s := \phi^j + s f^j_k \phi^{\bar k}, \quad j=1,\dots,m,
\]
where $f^j_k \in C^\infty(M)$ and $s \in \C$. Clearly $J(0)=J_0$ and, in light of the above remarks, $J(s)$ defines a curve of almost complex structures on $M$ as long as
\[
D(s)|_x := \det 
\begin{bmatrix}
\Id_m & s \Phi|_x \\
 & \\
 \bar s \bar \Phi|_x & \Id_m 
\end{bmatrix}
\neq 0, \quad \forall x \in M,
\]
where $\Phi$ is as in \cref{matrix:F}. This is always the case if $s$ is small enough, say $\abs{s} < \epsilon$, since $D(0) = 1$ is constant on $M$ and $D(s)$ depends continuously on $s$ (indeed, $D(s)$ is a polynomial in $s$ and $\bar s$ at every point $x \in M$).\\
As we remarked in \cref{sec:outline}, a disadvantage of defining $J_1$ as in \cref{eq:deformed} is that not all almost complex structures on $M$ can be obtained from a fixed $J_0$. One advantage of the choice we made is that, starting from $J_0$ and taking $J_1$ as in \eqref{eq:deformed}, one can immediately build a curve of almost complex structures $J(s)$, $\abs{s} < \epsilon$, along which we can control the rank of the Nijenhuis tensor.

\begin{theorem}\label{thm:lower:bound}
Let $(M,J_0)$ be a parallelizable almost complex manifold. Let $J_1$ be an almost complex structure defined by
\[
\omega^j := \phi^j + f^j_k \phi^{\bar k}, \quad j=1,\dots m,
\]
where $f^j_k \in C^\infty(M)$. Consider the curve of almost complex structures $J(s)$ defined by
\[
\omega^j_s := \phi^j + s f^j_k \phi^{\bar k}, \quad \abs{s} < \epsilon, \quad j=1,\dots,m.
\]
Then, at every point $x \in M$, we have that
\[
\rk N_{J(s)}|_x \ge \max \{\rk N_{J_0}|_x, \rk N_{J_1}|_x  \}.
\]
\end{theorem}

\begin{proof}
The lower bound $\rk N_{J(s)}|_x \ge \rk N_{J_0}|_x$ is a rephrasing of the well-known fact that the rank is a lower semi-continuous function. To prove that $\rk N_{J(s)}|_x \ge \rk N_{J_1}|_x$, observe that at each point $x \in M$, we can write 
\[
\bar \mu_s \omega^j_s = G^j_{kl} (s, \bar s) \, \omega ^{\bar k \bar l}_s, \quad k < l,
\]
where $G^j_{kl}$ are either identically vanishing or rational functions of $s$, $\bar s$ of the form
\[
G^j_{kl} = \frac{P (s, \bar s)}{D(s)^t},
\]
for some $T \in \N$.
Assume by contradiction that $\rk N_{J(s)}|_x < k_x := \rk N_{J_1}|_x$, for small $s$. Our assumption implies that the determinant of each $k_x \times k_x$ submatrix of $\left( G^j_{kl} (s, \bar s) \right)$ vanishes in a neighborhood of $0$. Since the determinants are polynomial in $G^j_{kl}$, they must be identically zero away form $D^{-1}(0)$. However, $J_1$ has precisely rank $k_x$, so that at least one of the determinants does not vanish for $s=1$, and it cannot be identically zero, giving a contradiction.
\end{proof}

\begin{cor}\label{cor:mni:nbhd}
Let $J_0$ be an almost complex structure on a parallelizable manifold $M$. Assume that there exists a maximally non-integrable almost complex structure $J_1$ on $M$ obtained as in \eqref{eq:deformed}. Then every neighborhood of $J_0$ contains a maximally non-integrable almost complex structure.
\end{cor}

\begin{remark}
In section \ref{sec:invariant}, we produce abundance of examples where Theorem \ref{thm:lower:bound} and Corollary \ref{cor:mni:nbhd} can be applied (see Proposition \ref{prop:jump}).
\end{remark}

For specific parallelizable manifolds and almost complex structures, we follow section \ref{sec:outline} and compute the operator $\bar \mu_s$ associated to a curve $J(s)$, hence its rank. As a consequence, we obtain a one-parameter version of the results proved in \cref{sec:holomorphic}.

\begin{proposition}\label{prop:curves}
Fix $k \in \{ 0,1,2,3 \}$. Let $X$ be any manifold among $T^6$, $\I$, $\Nak$. There exist curves of almost complex structures $J(s)$ on $X$ such that
\begin{itemize}
    \item [$(i)$] $J_0$ is integrable;
    
    \item [$(ii)$] $\bar \mu_s$ has constant rank $k$ for all $s$ such that $\abs{s} < \epsilon$.
\end{itemize}
In particular, each neighborhood of $J_0$ contains almost complex structures whose Nijenhuis tensor has arbitrary constant rank.
\end{proposition}

\begin{proof}
We show the computations for the maximally non-integrable almost complex structure on the Nakamura manifold. Consider the curve of almost complex structures on $\Nak$ defined by the co-frame of $(1,0)$-forms 
\[
\omega^1_s = \phi^1 + s\phi^{\bar 2} + s\phi^{\bar 3}, \quad \omega^2_s = \phi^2 + \frac{s}{2} \phi^{\bar 2}, \quad \omega^3_s = \phi^3 + \frac{s}{2} \phi^{\bar 3},
\]
obtained deforming the standard complex structure on $\Nak$ in the direction of the maximally non-integrable almost complex structure \eqref{eq:rank3:nakamura}. The same computations performed in section \ref{sec:holomorphic}, show that
\begin{align*}
    \bar \mu_s \omega^1_s &= -\frac{4s}{4-\abs{s}^2} \, \omega^{\bar 1 \bar 2} + \frac{4s}{4-\abs{s}^2} \, \omega^{\bar 1 \bar 3} + \frac{32 \abs{s}^2}{4-\abs{s}^2} \, \omega^{\bar 2 \bar 3}, \\
    & \\
    \bar \mu_s \omega^2_s &= \frac{2s}{4-\abs{s}^2} \, \omega^{\bar 1 \bar 2} + \frac{4 s^3 + 8 \abs{s}^2}{(4-\abs{s}^2)^2} \, \omega^{\bar 2 \bar 3}, \\
    & \\
    \bar \mu_s \omega^3_s &= \frac{2s}{4-\abs{s}^2} \, \omega^{\bar 1 \bar 3} - \frac{4 s^3 + 8 \abs{s}^2}{(4-\abs{s}^2)^2} \, \omega^{\bar 2 \bar 3},
\end{align*}
giving a maximally non integrable structure for all $s \neq 0$ such that $\abs{s} \neq 2$. Computations for the other cases proceed in a very similar fashion, replacing \eqref{eq:rank3:nakamura} with the appropriate almost complex structure from section \ref{sec:holomorphic}.
\end{proof}

\begin{remark}
Corollary \ref{cor:mni:nbhd}, together with the explicit structures provided in section \ref{sec:holomorphic}, gives an alternative proof of Proposition \ref{prop:curves} for the case $k = 3$.
\end{remark}

\section{Left-invariant almost complex structures on real \texorpdfstring{$6$}{}-nilmanifolds}\label{sec:invariant}

In this section we compute the rank of the Nijhenuis tensor of almost complex structures on $6$-dimensional, nilpotent (real) Lie algebras. As a consequence, for each $6$-nilmanifold we determine whether or not it admits a left-invariant almost complex structure whose Nijenhuis tensor has a given rank. If such a structure exists, we provide an explicit choice of complex parameters that allow to build it starting from an assigned almost complex structure. We also deduce a topological upper bound for the rank of $N_J$ on solvmanifolds obtained as a quotient of a completely solvable Lie group.
\vspace{.3cm}

We begin by proving our classification theorem for Lie algebras.

\begin{theorem}\label{thm:classification}
A $6$-dimensional, nilpotent, real Lie algebra admits an almost complex structure whose Nijenhuis tensor has rank $3$ if and only if it is isomorphic to one of
\[
\begin{array}{ll}
(0, 0, 12, 13, 14 + 23, 34 - 25),  & (0, 0, 12, 13, 14, 34 - 25),\\
(0, 0, 12, 13, 14 + 23, 24 + 15),  & (0, 0, 12, 13, 14, 23 + 15),\\
(0, 0, 12, 13, 23, 14),            & (0, 0, 12, 13, 23, 14 - 25),\\
(0, 0, 12, 13, 23, 14 + 25),       & (0, 0, 0, 12, 14 - 23, 15 + 34),\\
(0, 0, 0, 12, 14, 15 + 23),        & (0, 0, 0, 12, 14, 15 + 23 + 24),\\
(0, 0, 0, 12, 14, 15 + 24),        & (0, 0, 0, 12, 13, 14 + 35),\\
(0, 0, 0, 12, 23, 14 + 35),        & (0, 0, 0, 12, 23, 14 - 35),\\
(0, 0, 0, 12, 14, 24),             & (0, 0, 0, 12, 13 - 24, 14 + 23),\\
(0, 0, 0, 12, 14, 13 - 24),        & (0, 0, 0, 12, 13 + 14, 24),\\
(0, 0, 0, 12, 13, 14 + 23),        & (0, 0, 0, 12, 13, 24),\\
(0, 0, 0, 12, 13, 23).             &
\end{array}
\]
A $6$-dimensional, nilpotent, real Lie algebra does \textbf{not} admit an almost complex structure whose Nijenhuis tensor has rank $2$ if and only if it is isomorphic to one of 
\[
\begin{array}{ll}
(0, 0, 0, 12, 13, 23), & (0, 0, 0, 0, 0, 12 + 34),\\
(0, 0, 0, 0, 0, 12), & (0, 0, 0, 0, 0, 0).
\end{array}
\]
A $6$-dimensional, nilpotent, real Lie algebra does \textbf{not} admit an almost complex structure whose Nijenhuis tensor has rank $1$ if and only if it is isomorphic to one of
\[
\begin{array}{ll}
(0, 0, 12, 13, 14 + 23, 34 - 25), & (0, 0, 0, 0, 0, 0).
\end{array}
\]
\end{theorem}

The proof of the theorem is a collection of smaller results split among \cref{sec:rank3,sec:rank2,sec:rank1,sec:rank0}. In each of them, we deal with a different value for the rank. There are $34$ isomorphism classes of $6$-dimensional, nilpotent Lie algebras (\cite{Mag86}, see also \cite{Sal01}). We proceed to determine the rank of the almost complex structures existing on each of them. For the rest of the section we will work directly with almost complex structures defined on the elements of $\g^*$, adopting the corresponding notation.

\subsection{Structures of rank 3 (maximally non-integrable).}\label{sec:rank3}

$6$-dimensional, nilpotent Lie algebras for which $\ker d \cap A^1_\R$ is high-dimensional \emph{never} admit a maximally non-integrable almost complex structure.  This is a direct consequence of the following Lemma, that actually holds for any Lie algebra.

\begin{lemma}\label{lemma:bound}
Let $\g$ be a $2m$-dimensional Lie algebra and let $k = \dim_\R \left( \ker d \cap A^1_\R \right)$. Then for any almost complex structure $J$ on $\g$
\[
 \rk N_J \le 2m-k .
\]
\end{lemma}

\begin{proof}
Let $J$ be any almost complex structure on a $2m$-dimensional Lie algebra $\g$. We have that 
\begin{align*}
\rk N_J &= \dim_\C \left( \Ima \bar \mu \cap A^{0,2} \right) \le \dim_\C \left( \Ima d \cap A^2 \right) \\ &= \dim_\C \C \langle de^1, \dots, de^{2m} \rangle = 2m - k,
\end{align*} 
obtaining the desired upper bound.
\end{proof}

It follows immediately that several $6$-dimensional, nilpotent Lie algebras do not admit a maximally non-integrable almost complex structure.

\begin{cor}\label{cor:mni:nla}
Let $\g$ be any $6$-dimensional, nilpotent Lie algebra of the form
\begin{equation}\label{eq:NLA:starred}
\left( 0,0,0,0,\ast,\ast \right).
\end{equation}
Then $\g$ does not admit maximally non-integrable almost complex structures.
\end{cor}

There are three more $6$-dimensional, nilpotent Lie algebras that do not admit maximally non-integrable structures.

\begin{proposition}
None of the following Lie algebras admits a maximally non-integrable almost complex structure:
\begin{equation}\label{sporadic:mni}
\left( 0,0,0,12,13,14 \right), \quad \left( 0,0,0,12,14,15 \right), \quad \left( 0,0,12,13,14,15 \right).
\end{equation}
\end{proposition}

\begin{proof}
Let $J$ be an almost complex structure on $\g$, where $\g$ is any Lie algebra from \eqref{sporadic:mni}. Define the $(1,0)$-form
\[
\omega^1 := e^1 - i J e^1,
\]
and complete it to a basis of $(1,0)$-forms $\{ \omega^1,\omega^2,\omega^3 \}$. Noting that the differentials of $1$-forms can be written as $e^{1j}$, for some $j \in \{ 2,3,4,5 \}$, we have that
\[
\Ima d \cap A^2 = \C \langle de^1, \dots de^6 \rangle \subseteq e^1 \wedge A^1.
\]
Taking the projection on degree $(0,2)$, one can conclude that
\begin{align*}
\rk N_J &= \dim_\C \left( \Ima \bar \mu \cap A^{0,2} \right) \\
&\le \dim_\C \C \langle (e^1)^{0,1} \rangle \wedge (A^1)^{0,1} \\
&= \dim_\C \left( \omega^{\bar 1} \wedge A^{0,1} \right)=2,
\end{align*}
where $\left( \cdot \right) ^{0,1}$ denotes the projection on degree $(0,1)$.
\end{proof}

We complete the classification with the following existence result:

\begin{proposition}\label{existence:mni}
Any $6$-dimensional, nilpotent Lie algebra different from \eqref{eq:NLA:starred} and \eqref{sporadic:mni} admits a maximally non-integrable almost complex structure.
\end{proposition}

\begin{proof}
We prove the existence of a maximally non-integrable almost complex structure by explicitly exhibiting it. Let $\g = \R \langle e_1, \dots, e_6 \rangle$ be a nilpotent Lie algebra. Consider the almost complex structure $J_0$ defined by the co-frame of $(1,0)$-forms
\[
\phi^1 := e^1 + i e^2, \quad \phi^2 := e^3 + i e^4, \quad \phi^3 := e^5 + i e^6. 
\]
Let $J_1$ be the almost complex structure defined by the $(1,0)$-forms
\begin{align*}
\omega^1 &:= \phi^1 + e \, \phi^{ \bar 1} + f \, \phi^{\bar 2} + g \, \phi^{\bar 3}, \\
\omega^2 &:= \phi^2 + p \, \phi^{ \bar 1} + q \, \phi^{\bar 2} + r \, \phi^{\bar 3}, \\
\omega^3 &:= \phi^3 + s \, \phi^{ \bar 1} + t \, \phi^{\bar 2} + u \, \phi^{\bar 3},
\end{align*}
with $e,f,g,p,q,r,s,t,u$ complex parameters satisfying the condition $\det (P) \neq 0$, where $P$ is defined as in \eqref{matrix:P}. For the choice of parameters described in the second column of table \ref{table}, $J_1$ is a maximally non-integrable almost complex structure on the corresponding Lie algebra.
\end{proof}

\subsection{Structures of rank 2.}\label{sec:rank2}
The following proposition specifies which $6$-dimensional, nilpotent Lie algebras admit an almost complex structure of rank $2$. In its proof, as well as in other several instances in the rest of the section, we will use the following elementary fact:

\begin{lemma}\label{lemma:proportional}
Let $\g^* =\R \langle e^1, \dots, e^{2m} \rangle$ be the dual of a Lie algebra $\g$. Fix an almost complex structure $J$ on $\g^*$. Suppose that $\left( e^{jk} \right)^{0,2} =0$ for some indices $j,k \in \{1,\dots,2m \}$. Then $\left( e^k \right)^{0,1}$ is proportional to $\left( e^j \right)^{0,1}$.
\end{lemma}
\begin{proof}
Fix a basis of $(0,1)$-forms $\{ \omega^{\bar j} \}_{j=1}^m$, and write
\begin{align*}
    (e^j)^{0,1} &= A_1 \, \omega^{\bar 1} + \dots + A_m \, \omega^{\bar m}, \\
    (e^k)^{0,1} &= B_1 \, \omega^{\bar 1} + \dots + B_m \, \omega^{\bar m},
\end{align*}
with $A_j, B_j \in \C$. The condition $(e^{jk})^{0,2}=0$ implies that the matrix 
\[
\begin{bmatrix}
A_1 & \cdots & A_m \\
B_1 & \cdots & B_m
\end{bmatrix}
\]
has rank $1$. Since $(e^j)^{0,1}$ and $(e^k)^{0,1}$ cannot be the zero form because they are the projection on degree $(0,1)$ of a real form, the only possibility is that $(e^j)^{0,1}$ and $(e^k)^{0,1}$ are proportional to each other.
\end{proof}

\begin{proposition}\label{existence:rank:2}
Every $6$-dimensional, nilpotent Lie algebra different from
\begin{equation}\label{rank2}
\left( 0,0,0,12,13,23 \right), \, \left( 0,0,0,0,0,12+34 \right), \, \left( 0,0,0,0,0,12 \right), \, \left( 0,0,0,0,0,0 \right),
\end{equation}
admits an almost complex structure of rank $2$.
\end{proposition}

\begin{proof}
As a consequence of Lemma \ref{lemma:bound}, any Lie algebra among $\left( 0,0,0,0,0,12+34 \right)$, $\left( 0,0,0,0,0,12 \right)$, $\left( 0,0,0,0,0,0 \right)$ admits only structures of at most rank $1$.

Let $J$ be an almost complex structure on $(0, 0, 0, 12, 13, 23)$. We directly prove that $J$ cannot have rank $2$ by studying the $(0,2)$-degree part of the forms $e^{12}$, $e^{13}$, $e^{23}$.\\
First, suppose that $(e^{jk})^{0,2}=0$ for some $j,k \in \{ 1,2,3 \}$. Due to the symmetries in the indices $1,2,3$, and $4,5,6$ of the Lie algebra, we can assume that $(e^{12})^{0,2}=0$. By Lemma \ref{lemma:proportional}, $(e^1)^{0,1}$ and $(e^2)^{0,1}$ are proportional to the same $(0,1)$-form $\alpha$. Let $\{ \phi^j := e^j - i J e^j \}_{j=1}^6$ be a set of generators of $A^{1,0}$. If $(e^1)^{0,1}$ and $(e^2)^{0,1}$ are proportional to $\alpha$, then $\bar \mu \phi^j \in \C \langle \alpha \wedge (e^3)^{0,1} \rangle$, implying that $\rk N_J \le 1$.\\
Suppose now that $(e^{jk})^{0,2} \neq 0$, for all $ jk \in \{ 12, 13, 23 \}$. We further split the argument into three cases.

\textbf{Case 1:} two among $(e^{12})^{0,2}$, $(e^{13})^{0,2}$, $(e^{23})^{0,2}$ are multiple of each other. Again, for the symmetries of the Lie algebra we can assume that $(e^1)^{0,1} \wedge (e^2)^{0,1} = A (e^1)^{0,1} \wedge (e^3)^{0,1} $, with $A \in \C \setminus \{ 0 \}$. This implies that 
\[
(e^2)^{0,1} = A (e^3)^{0,1} + \gamma,
\]
where $\gamma$ is a $(0,1)$-form in the kernel of the map $(e^1)^{0,1} \wedge \bullet$. Set $\omega^{\bar 1} := e^1 + i J e^1$ and complete it to a basis of $(0,1)$-forms $\{ \omega^{\bar 1}$, $\omega^{\bar 2}$, $\omega^{\bar 3} \} $, so that 
\[
\gamma = C_1 \omega^{\bar 1} + C_2  \omega^{\bar 2} + C_3 \omega^{\bar 3}, \quad \text{with } C_j \in \C.
\]
The condition $(e^1)^{0,1} \wedge \gamma =0$ implies $C_2 = C_3 =0$, so that $\gamma$ is a multiple of $\omega^{\bar 1}$, and thus of $(e^1)^{0,1}$, since $\omega^{\bar 1} = 2 (e^1)^{0,1}$. This implies that $(e^{23})^{0,2}$ is a multiple of $(e^{12})^{0,2}$ (thus of $(e^{13})^{0,2}$):
\[
(e^2)^{0,1} \wedge (e^3)^{0,1} = \left( A (e^3)^{0,1} + \gamma \right) \wedge (e^3)^{0,1} = 2A C_1 (e^{13})^{0,2}.
\]
Since the $(e^{jk})^{0,2}$, for $ jk \in \{ 12, 13, 23 \}$, are all multiple of the same $(0,2)$-form, $N_J$ has at most rank $1$.

\textbf{Case 2:} one among $(e^{12})^{0,2}$, $(e^{13})^{0,2}$, $(e^{23})^{0,2}$ is a linear combination of the remaining two. By symmetry, we can assume that 
\[
(e^{12})^{0,2} = A (e^{13})^{0,2} + B (e^{23})^{0,2} = \left( A (e^1)^{0,1} + B (e^2)^{0,1}\right) \wedge (e^3)^{0,1},
\]
where $A$, $B \in \C$ are both non-zero, or else we would go back to the first case. Since $(e^{12})^{0,2} \neq 0$, set $\omega^{\bar j} := e^j + i J e^j$, $j = 1,2$, and complete to a basis of $(0,1)$-forms $\omega^{\bar 1}$, $\omega^{\bar 2}$, $\omega^{\bar 3}$. Proceeding as in the previous case, it is straightforward to see that
\[
(e^3)^{0,1} = C (e^1)^{0,1} + D (e^2)^{0,1},
\]
implying that the forms $(e^{13})^{0,2}$, $(e^{23})^{0,2}$ are both multiple of $(e^{12})^{0,2}$. The conclusion $\rk N_J \le 1$ follows as in the first case.

\textbf{Case 3:} all of the forms $(e^{12})^{0,2}$, $(e^{13})^{0,2}$, $(e^{23})^{0,2}$ are independent over $\C$.\\
We prove that $\bar \mu$ has necessarily rank $3$. Consider the $(1,0)$-forms $\phi^j := e^j - i J e^j $, $j=1,2,3$. The projections on degree $(0,2)$ of the forms $e^{jk}$ can be expressed in terms of the $\phi^j$ as
\[
(e^{12})^{0,2} = \frac{1}{4} \phi^{\bar 1 \bar 2}, \quad (e^{13})^{0,2} = \frac{1}{4} \phi^{\bar 1 \bar 3}, \quad (e^{23})^{0,2} = \frac{1}{4} \phi^{\bar 2 \bar 3},
\]
that, by assumption, are independent over $\C$. This implies that also the $\phi^j$, $j=1,2,3$ are independent, and so they are a basis of $(1,0)$-forms. In terms of the basis $e^j$, $J$ has the expression
\[
J=
\begin{bmatrix}
A & B \\
C & D
\end{bmatrix}.
\]
From $J^2 = -\Id$, we obtain the relation
\begin{equation}\label{key}
A^2 + BC = -\Id,
\end{equation}
Computing the rank of $\bar \mu$ amounts to compute the rank of the matrix $B$, since
\[
\bar \mu \phi^j = - \frac{i}{4} \left( B_{j1} \phi^{\bar 1 \bar 2} + B_{j2} \phi^{\bar 1 \bar 3} + B_{j3} \phi^{\bar 2 \bar 3} \right) , \quad j = 1,2,3.
\]
Now, the complex $1$-forms $\{ \phi^j, \phi^{\bar j} \}_{j=1}^3$ can be written in terms of the $e^j$ as 
\[
\left(
\phi^1,
\phi^2,
\phi^3,
\phi^{\bar 1},
\phi^{\bar 2},
\phi^{\bar 3} \right)^t
=
Q 
\left( e^1,
e^2,
e^3,
e^4,
e^5,
e^6 \right)^t,
\]
where $Q$ is the block matrix
\[
Q=
\begin{bmatrix}
\Id - iA & -iB \\
\\
\Id + i A & iB
\end{bmatrix}.
\]
Since $\{ \phi^j, \phi^{\bar j} \}_{j=1}^3$ is a basis of $1$-forms, $Q$ must be invertible and the matrix 
\[
Q^* Q =
\begin{bmatrix}
2 \left( \Id - iA \right) \left( \Id + iA \right) & i \left[ \left(I-iA \right) B - \left(I+iA \right) B  \right] \\
-i \left[ B \left(I+iA \right) - B \left(I-iA \right)  \right] & B^2 
\end{bmatrix}
\]
is positive-definite. In particular its principal minor
\[
\left( \Id + iA \right) \left( \Id - iA \right) = \Id + A^2 
\]
is also positive-definite, and by \cref{key} it is equal to $BC$. This forces $B$, and thus $\bar \mu$, to have rank $3$, concluding the proof of the fact that $(0, 0, 0, 12, 13, 23)$ has no almost complex structure whose Nijenhuis tensor has precisely rank $2$.

Finally, we prove that each of the remaining $6$-dimensional, nilpotent Lie algebras admits an almost complex structure of rank $2$ by explicitly exhibiting it, as we did in the proof of Proposition \ref{existence:mni}. The explicit choice of constants providing the desired structure can be found in the third column of \cref{table}.
\end{proof}

\subsection{Structures of rank 1.}\label{sec:rank1}

The only $6$-dimensional nilpotent Lie algebras not admitting an almost complex structure of rank $1$ are 
\begin{equation}\label{exception:rank1}
(0, 0, 0, 0, 0, 0) \quad \text{and} \quad (0, 0, 12, 13, 14 + 23, 34 - 25).
\end{equation}
For the former, this is again an immediate consequence of Lemma \ref{lemma:bound}, while for the latter it is the content of the following

\begin{proposition}\label{prop:norank1}
Any almost complex structure on the Lie algebra
\[
\g := (0, 0, 12, 13, 14 + 23, 34 - 25)
\]
has at least rank $2$. 
\end{proposition}

The proof of Proposition \ref{prop:norank1} proceeds assuming the existence of a structure of rank $1$ on $\g$ and splitting the argument into four main cases, according to whether or not the projection on degree $(0,2)$ of $d$-exact $2$-forms is zero. In each of the cases, we reach an absurd by contradicting the following well-known Lemma on the linear algebra of almost complex structures (or a direct consequence of it), of which we give a proof for the sake of completeness.

\begin{lemma}\label{lemma:independent}
Let $V$ be a $2m$-dimensional real vector space and let $\{ e_j \}_{j=1}^{2m}$ be a basis of $V$. Fix an almost complex structure $J$ on $V$ and consider the projection $\pi^{0,1} \colon V^\C \rightarrow V^{0,1}$ . Then the subspace
\[
S := \C \langle \pi^{0,1}(e_j), \pi^{0,1}(e_k), \pi^{0,1}(e_l) \rangle, \quad j \neq k \neq l,
\]
has at least complex dimension $2$.
\end{lemma}

\begin{proof}
Summing over repeated indices, we can write $J e_j = J_j^k e_k$, $J_j^k \in \R$. Since $e_j$ is a real vector, its $(0,1)$-part cannot be zero and is given by
\[
\pi^{0,1}(e_j) = \frac{1}{2} \left( e_j + i J e_j \right).
\]
Thus $S$ is at least $1$-dimensional. Assume by contradiction that $S$ has precisely dimension $1$. Then the vectors $( e_k + i J e_k )$ and $( e_l + i J e_l )$ are proportional, up to non-vanishing complex constants, to $( e_j + i J e_j )$, i.e.,
\[
( e_j + i J e_j)  = A ( e_k + i J e_k)  = B ( e_l + i J e_l), \text{ with } A, B \in \C \setminus \{0\}. 
\]
Writing explicitly the vectors in terms of the element of the basis, we obtain a system of equations involving the coefficients $A$, $B$, $J_j^k$. We are interested in the part involving only $e_j$ and $e_k$:
\[
\begin{cases}
1 + i J_j^j = i A J_j^k = i B J_j^l, \\
\\
i J_k^j = A (1 + i J_k^k) = i B J_k^l.
\end{cases}
\]
From the first equation we have that $J_j^k \neq 0$, $J_j^l \neq 0$, obtaining the value for the constants 
\[
A = -i \frac{(1 + i J_j^j)}{J_j^k}, \quad B = -i \frac{(1 + i J_j^j)}{J_j^l},
\]
and, substituting in the second equation, we are left with
\[
i J_k^j = -i \frac{(1 + i J_j^j)(1 + i J_k^k)}{J_j^k} = \frac{J_k^l}{J_j^l}(1 + i J_j^j).
\]
From the equality $i J_k^j = J_k^l / J_j^l (1 + i J_j^j)$, we deduce that both $J_k^l$ and $J_k^j$ must vanish. The remaining equation $(1 + i J_j^j)(1 + i J_k^k) = 0$ leads to the contradiction $(J_j^j)^2 = -1$, concluding the proof.
\end{proof}

We will also need the following lemma:

\begin{lemma}\label{lemma:dependent}
Let $\g^* =\R \langle e^1, \dots, e^{2m} \rangle$ be the dual of a Lie algebra $\g$. Fix an almost complex structure $J$ on $\g^*$. Suppose that $\left( e^k \right)^{0,1}$ is proportional to $\left( e^j \right)^{0,1}$ for some $j,k \in \{1,\dots, 2m \}$. Then $J e^j$, $J e^k \in \R \langle e^j, e^k \rangle$.
\end{lemma}

\begin{proof}
With the same notation from the proof of Lemma \ref{lemma:independent}, we have, passing to the dual, that 
\[
( e^j )^{0,1} = - \frac{i}{2} \left( J^j_p + i \delta^j_p \right) e^p.
\]
The forms $( e^j )^{0,1}$ and $( e^k )^{0,1}$ are multiple of each other if and only if the matrix
\[
\begin{bmatrix}
J^j_1 & \cdots & J^j_j +i & \cdots & J^j_k & \cdots & J^j_{2m} \\
J^k_1 & \cdots & J^k_j & \cdots & J^k_k +i & \cdots & J^k_{2m}
\end{bmatrix}
\]
has rank $1$. Imposing that the determinant of each of its $2 \times 2$ minors vanishes and separating real and imaginary part, we see that $J$ must satisfy the following relations
\begin{equation}\label{eq:proportional}
    \begin{cases}
    J^k_k = - J^j_j & \\[2pt]
    (J^j_j)^2 + J^j_k J^k_j = -1 & \\[2pt]
    J^j_p = J^k_p=0 & \text{if $p \notin \{ j,k \}$.}
    \end{cases}
\end{equation}
Our claim follows from \eqref{eq:proportional}, since $Je^j = J^j_j e^j + J^j_k e^k$ and $Je^k = J^k_j e^j - J^j_j e^k$.
\end{proof}

We can now prove that any almost complex structure on $\g$ has at least rank $2$.

\begin{proof}[Proof of Proposition \ref{prop:norank1}]
There are no complex structures on $\g$ \cite{Sal01}.
Let $J$ be an almost complex structure on $\g$ and assume by contradiction that it has rank $1$. The $(1,0)$-forms $\{ \phi^j := e^j - i J e^j \}_{j=1}^6$ are a set of generators of $A^{1,0}$, so we can compute the rank of $\bar \mu$ focusing only on the $\phi^j$. We have that
\[
\bar \mu \phi^j = -i \left( J^j_k + i \delta^j_k \right) \left( de^k \right)^{0,2}.
\]
The rank of $\bar \mu$ is the rank of a suitable submatrix obtained removing from $J + i \Id_6$ the columns corresponding to the indices for which $\left( de^k \right)^{0,2} =0$, and taking linear combinations of such columns if $\left( de^j \right)^{0,2}$ is a non-zero multiple of $\left( de^k \right)^{0,2}$. Since $de^1 = de^2 =0$, we focus on a $6 \times 4$ matrix, identifying for each case which among $\left( de^k \right)^{0,2}$ vanish and which are proportional to each other.

\textbf{Case A: $\mathbf{(e^{12})^{0,2}=(e^{13})^{0,2}=0}$.}\\
By Lemma \ref{lemma:proportional}, $(e^{12})^{0,2}=0$ implies that $(e^1)^{0,1}$ and $(e^2)^{0,1}$ are proportional. Similarly, $(e^1)^{0,1}$ is proportional to $(e^3)^{0,1}$, contradicting Lemma \ref{lemma:independent}.

\textbf{Case B: $\mathbf{(e^{12})^{0,2} = 0}$ and $\mathbf{(e^{13})^{0,2} \neq 0}$.}\\
Again by Lemma \ref{lemma:proportional}, $(e^2)^{0,1} = A (e^1)^{0,1}$ for some $A \in \C \setminus \{ 0 \}$. Consider the $2$-form 
\[
\alpha := (e^{14 + 23})^{0,2} = (e^1)^{0,1} \wedge \left( (e^4)^{0,1} + A (e^3)^{0,1}  \right).
\]
If $\alpha$ is not a multiple of $(e^{13})^{0,2}$, then $(e^4)^{0,1}$ is independent of $(e^1)^{0,1}$ and $(e^3)^{0,1}$, giving a basis of $(1,0)$-forms $\{ \phi^1, \phi^3, \phi^4 \}$. In terms of such a basis, we have that
\[
(e^{13})^{0,2} = \frac{1}{4} \phi^{\bar 1 \bar 3}, \quad (e^{14 + 23})^{0,2} = \frac{1}{4} ( \phi^{\bar 1 \bar 4} + A \phi^{\bar 1 \bar 3}),
\]
and
\[
(e^{34 - 25})^{0,2} = \frac{1}{4} \phi^{\bar 3 \bar 4} + \phi^{\bar 1} \wedge \theta,
\]
for some $(0,1)$-form $\theta$. These $(0,2)$-forms are independent, thus the rank of $\bar \mu$ is determined by the corresponding columns of $J + i \Id_6$, i.e., by the rank of the $6 \times 3$ matrix $ U := ( J^j_k + i \delta^j_k )$, $j =1,\dots, 6$, $k =4,5,6$. Since $\bar \mu$ has rank $1$, so does $U$, and we can apply repeatedly Lemma \ref{lemma:dependent} to its columns to deduce the condition $J^3_3 = -i$, reaching an absurd.\\
The proof of case B is concluded if we prove that $\alpha$ cannot be a multiple of $(e^{13})^{0,2}$. Assume by contradiction that $(e^{14 + 23})^{0,2}$ is a multiple of $(e^{13})^{0,2}$. Then necessarily $(e^4)^{0,1}$ is a linear combination
\[
(e^4)^{0,1} = B (e^1)^{0,1} + C (e^3)^{0,1}.
\]
If $B \neq 0$, we can redefine the elements of the basis $e^j$ setting
\[
\hat{e}^4 = e^4 - C e^3,
\]
so that $( \hat{e}^4 )^{0,1} = B (e^1)^{0,1} = B/A \, (e^2)^{0,1}$. This contradicts Lemma \ref{lemma:independent}. \\
If $B=0$, then $ (e^4)^{0,1} = C (e^3)^{0,1}$, giving a simple expression for the projection on degree $(0,2)$ of $de^6$: 
\[
\beta := (e^{34 - 25})^{0,2} = - A (e^1)^{0,1} \wedge (e^5)^{0,1}.
\]
If $\beta$ is proportional to $(e^{13})^{0,2}$, then 
\[
(e^5)^{0,1} = D (e^1)^{0,1} + E (e^3)^{0,1}.
\]
As above, if $D=0$, we immediately get a contradiction to Lemma
\ref{lemma:independent}. The same follows when $E \neq 0$, by redefining $\hat e ^5 = e^5 -E e^3$.\\
If $\beta$ is not proportional to $(e^{13})^{0,2}$, then the matrix that determines the rank of $\bar \mu$ is
\[
\begin{bmatrix}
0 & 0 & J^3_4 & J^4_4 + i & J^5_4 & J^6_4 \\
0 & 0 & J^3_6 & J^4_6 & J^5_6 & J^6_6 +i
\end{bmatrix}^T.
\]
By Lemma \ref{lemma:dependent} applied to $ (e^3)^{0,1}$ and $(e^4)^{0,1}$, it must be $J^3_6 = J^4_6=0$. Imposing the condition $\rk \bar \mu = 1$, we easily obtain the contradiction $J^3_3 = -i$, proving that $\alpha$ is not a multiple of $(e^{13})^{0,2}$, and thus proving our claim in case B.

\textbf{Case C: $\mathbf{(e^{12})^{0,2} \neq 0}$.}\\
The proof is similar to that of case B, with slightly longer computations.
\end{proof}

An existence result completes the classification.

\begin{proposition}\label{existence:rank1}
Every $6$-dimensional, nilpotent Lie algebra different from \eqref{exception:rank1} admits an almost complex structure of rank $1$.
\end{proposition}
\begin{proof}
The existence of a structure of rank $1$ is proved as in Proposition \ref{existence:mni}. The constants providing the desired structure are presented in the fourth column of \cref{table}.
\end{proof}

\subsection{Complex structures.}\label{sec:rank0}

The classification of $6$-dimensional, nilpotent Lie algebras admitting a complex structure was carried out by Salamon \cite{Sal01}. For the sake of completeness, in the last column of \cref{table} we give explicit constants that allow to obtain examples of complex structures following the idea given in the proof of Proposition \ref{existence:mni}. The only Lie algebras on which a complex structure cannot be obtained in this way are $(0, 0, 0, 12, 23, 14 - 35)$ and $(0, 0, 0, 0, 12, 14 + 25)$. In these two cases, it is immediate to check that the co-frame of $(1,0)$-forms
\[
\phi^1 := e^1 + i e^3, \quad \phi^2 := e^4 + i e^5, \quad \phi^3 := -e^2 + i e^6
\]
defines a complex structure on $(0, 0, 0, 12, 23, 14 - 35)$, while the co-frame
\[
\phi^1 := e^1 + i e^2, \quad \phi^2 := e^4 + i e^5, \quad \phi^3 := e^3 + i e^6
\]
defines a complex structure on $(0, 0, 0, 0, 12, 14 + 25)$.

\begin{remark}
Complex structures on $6$-dimensional, nilpotent Lie algebras have two types of canonical basis \cite[Theorem 2.5]{Sal01}. Type $(I)$ has the form
\[
\omega^1 = e^1 - i e^2, \quad \omega^1 = e^3 - i e^4, \quad \omega^1 = e^5 - i e^6,
\]
while type $(II)$ can be written as
\[
\omega^1 = e^1 - i e^2, \quad \omega^1 = e^4 - i e^5, \quad \omega^1 = e^3 - i e^6.
\]
We point out that when we give examples of explicit structures, we are deforming a (possibly non-integrable) structure with a basis of type $(I)$. The Lie algebras on which a complex structure cannot be obtained in this way are precisely those that admit only complex structures of type $(II)$ \cite{Sal01}.
\end{remark}

\subsection{Consequences on homogeneous manifolds.}\label{sec:homogeneous}

The classification by rank of almost complex structures on $6$-dimensional, nilpotent Lie algebras allows to establish which $6$-nilmanifold admit a left-invariant almost complex structure of a certain rank.

\begin{theorem}\label{thm:nilmanifolds}
Let $M= \Gamma \backslash G$ be a $6$-nilmanifold and let $\g$ be the Lie algebra of $G$. Then $M$ admits a left-invariant almost complex structure of rank $k$ if and only if $\g$ admits an almost complex structure of rank $k$, according to Theorem \ref{thm:classification}.
\end{theorem}

\begin{proof}
Fix an almost complex structure $J$ on $M$. Our claim follows from the classification given in Theorem \ref{thm:classification} and the usual bijection between left-invariant almost complex structures on $M$ and almost complex structures on $\g$, after noting that the rank of $N_J$ on $M$ is equal to the rank of the almost complex structure induced by $J$ on $\g$.
\end{proof}

More in general, in the following result we establish a topological upper bound for the rank of left-invariant almost complex structures on certain solvmanifolds of arbitrary dimension.

\begin{theorem}\label{thm:solvmanifold}
Let $M = \Gamma \backslash G$ be a solvmanifold. Assume that $G$ is a completely solvable Lie group. Let $J$ be a left-invariant almost complex structure on $M$. Then
\[
\rk N_J \le \dim_\R M - b_1(M).
\]

\end{theorem}

\begin{proof}
Let $\g$ be the Lie algebra of $G$. Given any left-invariant almost complex structure $J$ on $M$, there is a corresponding almost complex structure $\Tilde{J}$ on $\g$, and $\rk N_J = \rk N_{\Tilde{J}}$. Let $b_1 (\g)$ be the real dimension of $H^1(\g) := \left( \ker d \cap A^1 \right) / \left( \Ima d \cap A^1 \right) = \left( \ker d \cap A^1 \right)$.  By Lemma \ref{lemma:bound},
\[
\rk N_J = \rk N_{\Tilde{J}} \le 2m - \dim _\R \left( \ker d \cap A^1_\R \right) = 2m - b_1 (\g).
\]
By Hattori's theorem \cite{Hat60}, there is an isomorphism
\[
H^\bullet (\g) \cong H^\bullet_{dR} (M;\R),
\]
thus $b_1(\g) = b_1 (M)$.
\end{proof}

We exploit our explicit examples to show that all the possible jumps in the local value of the rank of the Nijenhuis tensor, provided they satisfy the necessary condition of lower semi-continuity, occur on some $6$-nilmanifold.

\begin{proposition}\label{prop:jump}
For all $k_0 \le k_1 \in \{ 0,1,2,3 \}$, there exist a $6$-nilmanifold $M$, depending on $k_0,k_1$, and a left-invariant almost complex structure $J_0$ on $M$ such that 
\[
\rk N_{J_0} = k_0,
\]
and in any neighborhood of $J_0$ there are left-invariant almost complex structures of rank $k_1$.
\end{proposition}

\begin{proof}
The proof follows immediately applying the same idea of Proposition \ref{prop:curves} to the explicit almost complex structures given in \cref{table}.
\end{proof}

\section{Table of the possible ranks}\label{sec:table}
This section contains the table summarizing the possible ranks of almost complex structures on $6$-dimensional, nilpotent Lie algebras. The first column lists the $34$ possible isomorphism type of Lie algebras. The remaining columns list whether or not a structure of prescribed rank exists on each of them. When such a structure exists, the choice of parameters
\[
\begin{bmatrix}
e & f & g \\
p & q & r \\
s & t & u \\
\end{bmatrix}
\]
provided in the table, allows to obtain it starting from a fixed almost complex structure (see proof of Proposition \ref{existence:mni}). By the word \emph{generic}, we mean that a generic almost complex structure will have the corresponding rank. The non-existence of structures is proved in Section \ref{sec:invariant}.

\begin{remark}
The computations for the rank of the almost complex structure presented in \cref{table} have been checked using \emph{Wolfram Mathematica, Version 13.1}.
\end{remark}

\clearpage

{
\setlength{\LTleft}{-20cm plus -1fill}
\setlength{\LTright}{\LTleft}

\begin{longtable}{|C|C|C|C|C|}
\caption{Rank of almost complex structures}\label{table} \\
\hline & & & & \\[-10pt]

\text{Lie Algebra} & \text{rank 3} & \text{rank 2} & \text{rank 1} & \text{rank 0} \\

\hline

(0,0,12,13,14+23,34-25) & 
\begin{array}{c}
\text{generic} \\
{
\tiny
\begin{bmatrix}
0 & 0 & 1 \\
0 & 0 & 0 \\
0 & 0 & 0 \\
\end{bmatrix}
}
\end{array}
& 
\begin{array}{c}
\text{yes}\\
{
\tiny
\begin{bmatrix}
0 & 0 & 0 \\
0 & 0 & 0 \\
0 & 0 & 0 \\
\end{bmatrix}
}
\end{array} 
& \text{no} & \text{no} \\[20pt]

\hline

(0, 0, 12, 13, 14, 34 - 25) & 
\begin{array}{c}
\text{generic} \\
{
\tiny
\begin{bmatrix}
0 & 0 & 1 \\
0 & 0 & 0 \\
0 & 2 & 0 \\
\end{bmatrix}
}
\end{array}
& 
\begin{array}{c}
\text{yes}\\
{
\tiny
\begin{bmatrix}
0 & 0 & 0 \\
0 & 0 & 0 \\
0 & 0 & 0 \\
\end{bmatrix}
}
\end{array} 
& 
\begin{array}{c}
\text{yes}\\
{
\tiny
\begin{bmatrix}
0 & -\frac{1}{2} & -\frac{1}{4}+ \frac{i}{2} \\
0 & 0 & -\frac{1}{2} \\
0 & 0 & 0 \\
\end{bmatrix}
}
\end{array}
& 
\text{no} \\[20pt]
\hline
(0, 0, 12, 13, 14, 15) & \text{no} 
& 
\begin{array}{c}
\text{generic}\\
{
\tiny
\begin{bmatrix}
0 & 0 & 0 \\
0 & 0 & 0 \\
0 & 0 & 0 \\
\end{bmatrix}
}
\end{array} 
& 
\begin{array}{c}
\text{yes}\\
{
\tiny
\begin{bmatrix}
0 & 0 & 0 \\
0 & 1 & 1 \\
0 & -4 & 1 \\
\end{bmatrix}
}
\end{array} 
& 
\text{no} \\[20pt]
\hline
(0, 0, 12, 13, 14 + 23, 24 + 15) & 
\begin{array}{c}
\text{generic}\\
{
\tiny
\begin{bmatrix}
0 & 0 & 1 \\
0 & 0 & 0 \\
0 & 0 & 0 \\
\end{bmatrix}
}
\end{array} 
& 
\begin{array}{c}
\text{yes}\\
{
\tiny
\begin{bmatrix}
0 & 0 & 0 \\
0 & 0 & 0 \\
0 & 0 & 0 \\
\end{bmatrix}
}
\end{array} 
& 
\begin{array}{c}
\text{yes}\\
{
\tiny
\begin{bmatrix}
0 & 0 & 0 \\
0 & 1 & 1 \\
0 & -2 + 2 \sqrt{2} & 1 \\
\end{bmatrix}
}
\end{array} 
& 
\text{no} \\[20pt]
\hline
(0, 0, 12, 13, 14, 23 + 15) 
&
\begin{array}{c}
\text{generic}\\
{
\tiny
\begin{bmatrix}
0 & 0 & 1 \\
0 & 0 & 0 \\
0 & 0 & 0 \\
\end{bmatrix}
}
\end{array}  
& 
\begin{array}{c}
\text{yes}\\
{
\tiny
\begin{bmatrix}
0 & 0 & 0 \\
0 & 0 & 0 \\
0 & 0 & 0 \\
\end{bmatrix}
}
\end{array} 
& 
\begin{array}{c}
\text{yes}\\
{
\tiny
\begin{bmatrix}
0 & 0 & 0 \\
0 & 0 & 2 \\
0 & 2 & 1 \\
\end{bmatrix}
}
\end{array} 
& 
\text{no} \\[20pt]
\hline
(0, 0, 12, 13, 23, 14) 
& 
\begin{array}{c}
\text{generic}\\
{
\tiny
\begin{bmatrix}
0 & 0 & 1 \\
0 & 0 & 0 \\
0 & 0 & 0 \\
\end{bmatrix}
}
\end{array} 
& 
\begin{array}{c}
\text{yes}\\
{
\tiny
\begin{bmatrix}
0 & 0 & 0 \\
0 & 0 & 1 \\
0 & 0 & 0 \\
\end{bmatrix}
}
\end{array} 
& 
\begin{array}{c}
\text{yes}\\
{
\tiny
\begin{bmatrix}
0 & 0 & 0 \\
0 & 0 & 0 \\
0 & 0 & 0 \\
\end{bmatrix}
}
\end{array} 
& \text{no} \\[20pt]
\hline
(0, 0, 12, 13, 23, 14 - 25) 
& 
\begin{array}{c}
\text{generic}\\
{
\tiny
\begin{bmatrix}
0 & 0 & 1 \\
0 & 0 & 0 \\
0 & 0 & 0 \\
\end{bmatrix}
}
\end{array} 
& 
\begin{array}{c}
\text{yes}\\
{
\tiny
\begin{bmatrix}
0 & 0 & 0 \\
0 & 0 & 0 \\
0 & 0 & 0 \\
\end{bmatrix}
}
\end{array} 
& 
\begin{array}{c}
\text{yes}\\
{
\tiny
\begin{bmatrix}
0 & 0 & 0 \\
0 & 0 & 2 \\
0 & 2 & 1 \\
\end{bmatrix}
}
\end{array}& 
\text{no} \\[20pt]
\hline
(0, 0, 12, 13, 23, 14 + 25) 
& 
\begin{array}{c}
\text{generic}\\
{
\tiny
\begin{bmatrix}
0 & 0 & 1 \\
0 & 0 & 0 \\
0 & 0 & 0 \\
\end{bmatrix}
}
\end{array}& 
\begin{array}{c}
\text{yes}\\
{
\tiny
\begin{bmatrix}
0 & 0 & 0 \\
0 & 0 & 0 \\
0 & 0 & 0 \\
\end{bmatrix}
}
\end{array} 
& 
\begin{array}{c}
\text{yes}\\
{
\tiny
\begin{bmatrix}
0 & 0 & 0 \\
0 & 1 & -1 \\
0 & 2 & 1 \\
\end{bmatrix}
}
\end{array} 
& 
\begin{array}{c}
\text{yes}\\
{
\tiny
\begin{bmatrix}
0 & 0 & 0 \\
0 & 1 & -1 \\
0 & 2 & i \\
\end{bmatrix}
}
\end{array} 
\\[20pt]

\hline
(0, 0, 0, 12, 14 - 23, 15 + 34) 
& 
\begin{array}{c}
\text{generic}\\
{
\tiny
\begin{bmatrix}
0 & 0 & 1 \\
0 & 0 & 0 \\
0 & 1 & 0 \\
\end{bmatrix}
}
\end{array} 
& 
\begin{array}{c}
\text{yes}\\
{
\tiny
\begin{bmatrix}
0 & 0 & 0 \\
0 & 0 & 1 \\
0 & 0 & 0 \\
\end{bmatrix}
}
\end{array} 
& 
\begin{array}{c}
\text{yes}\\
{
\tiny
\begin{bmatrix}
0 & 0 & 0 \\
0 & 0 & 0 \\
0 & 0 & 0 \\
\end{bmatrix}
}
\end{array} 
& 
\text{no} \\[20pt]

\hline
(0, 0, 0, 12, 14, 15 + 23) 
& 
\begin{array}{c}
\text{generic}\\
{
\tiny
\begin{bmatrix}
0 & 0 & 1 \\
0 & 0 & 0 \\
0 & 0 & 0 \\
\end{bmatrix}
}
\end{array} 
& 
\begin{array}{c}
\text{yes}\\
{
\tiny
\begin{bmatrix}
0 & 0 & 0 \\
0 & 0 & 1 \\
0 & 0 & 0 \\
\end{bmatrix}
}
\end{array} 
& 
\begin{array}{c}
\text{yes}\\
{
\tiny
\begin{bmatrix}
0 & 0 & 0 \\
0 & 0 & 0 \\
0 & 0 & 0 \\
\end{bmatrix}
}
\end{array} 

& \text{no} \\[20pt]

\hline
(0, 0, 0, 12, 14, 15 + 23 + 24) 
& 
\begin{array}{c}
\text{generic}\\
{
\tiny
\begin{bmatrix}
0 & 0 & 1 \\
0 & 0 & 0 \\
0 & 0 & 0 \\
\end{bmatrix}
}
\end{array} 
& 
\begin{array}{c}
\text{yes}\\
{
\tiny
\begin{bmatrix}
0 & 0 & 0 \\
0 & 0 & 1 \\
0 & 0 & 0 \\
\end{bmatrix}
}
\end{array} 
& 
\begin{array}{c}
\text{yes}\\
{
\tiny
\begin{bmatrix}
0 & 0 & 0 \\
0 & 0 & 0 \\
0 & 0 & 0 \\
\end{bmatrix}
}
\end{array} 
& \text{no} \\[20pt]

\hline
(0, 0, 0, 12, 14, 15 + 24) 
& 
\begin{array}{c}
\text{generic}\\
{
\tiny
\begin{bmatrix}
0 & 0 & 1 \\
0 & 0 & 0 \\
0 & 0 & 0 \\
\end{bmatrix}
}
\end{array} 
& 
\begin{array}{c}
\text{yes}\\
{
\tiny
\begin{bmatrix}
0 & 0 & 0 \\
0 & 0 & 1 \\
0 & 0 & 0 \\
\end{bmatrix}
}
\end{array} 
& 
\begin{array}{c}
\text{yes}\\
{
\tiny
\begin{bmatrix}
0 & 0 & 0 \\
0 & 0 & 0 \\
0 & 0 & 0 \\
\end{bmatrix}
}
\end{array} 
& 
\text{no} \\[20pt]

\hline
(0, 0, 0, 12, 14, 15) 
& 
\text{no} 
& 
\begin{array}{c}
\text{generic}\\
{
\tiny
\begin{bmatrix}
0 & 0 & 0 \\
0 & 0 & 1 \\
0 & 0 & 0 \\
\end{bmatrix}
}
\end{array} 
& 
\begin{array}{c}
\text{yes}\\
{
\tiny
\begin{bmatrix}
0 & 0 & 0 \\
0 & 0 & 0 \\
0 & 0 & 0 \\
\end{bmatrix}
}
\end{array} 
& 
\text{no} \\[20pt]

\hline
(0, 0, 0, 12, 13, 14 + 35) 
& 
\begin{array}{c}
\text{generic}\\
{
\tiny
\begin{bmatrix}
0 & 0 & 1 \\
0 & 0 & 0 \\
0 & 0 & 0 \\
\end{bmatrix}
}
\end{array} 
& 
\begin{array}{c}
\text{yes}\\
{
\tiny
\begin{bmatrix}
0 & 0 & 0 \\
0 & 0 & 1 \\
0 & 0 & 0 \\
\end{bmatrix}
}
\end{array} 
& 
\begin{array}{c}
\text{yes}\\
{
\tiny
\begin{bmatrix}
0 & 0 & 0 \\
0 & 0 & 0 \\
0 & 0 & 0 \\
\end{bmatrix}
}
\end{array} 
& 
\text{no} \\[20pt]

\hline
(0, 0, 0, 12, 23, 14 + 35) 
& 
\begin{array}{c}
\text{generic}\\
{
\tiny
\begin{bmatrix}
0 & 0 & 1 \\
0 & 0 & 0 \\
0 & 0 & 0 \\
\end{bmatrix}
}
\end{array} 
& 
\begin{array}{c}
\text{yes}\\
{
\tiny
\begin{bmatrix}
0 & 0 & 0 \\
0 & 0 & 1 \\
0 & 0 & 0 \\
\end{bmatrix}
}
\end{array} 
& 
\begin{array}{c}
\text{yes}\\
{
\tiny
\begin{bmatrix}
0 & 0 & 0 \\
0 & 0 & 0 \\
0 & 0 & 0 \\
\end{bmatrix}
}
\end{array} 
& 
\text{no} \\[20pt]

\hline
(0, 0, 0, 12, 23, 14 - 35) 
& 
\begin{array}{c}
\text{generic}\\
{
\tiny
\begin{bmatrix}
0 & 0 & 1 \\
0 & 0 & 0 \\
0 & 0 & 0 \\
\end{bmatrix}
}
\end{array} 
& 
\begin{array}{c}
\text{yes}\\
{
\tiny
\begin{bmatrix}
0 & 0 & 0 \\
0 & 0 & 1 \\
0 & 0 & 0 \\
\end{bmatrix}
}
\end{array} 
& 
\begin{array}{c}
\text{yes}\\
{
\tiny
\begin{bmatrix}
0 & 0 & 0 \\
0 & 0 & 0 \\
0 & 0 & 0 \\
\end{bmatrix}
}
\end{array} 
& 
\begin{array}{c}
  \text{yes: see \cref{sec:rank0}}
\end{array} \\[20pt]

\hline
(0, 0, 0, 12, 14, 24) 
& 
\begin{array}{c}
\text{generic}\\
{
\tiny
\begin{bmatrix}
0 & 0 & 1 \\
0 & 0 & 0 \\
0 & 0 & 0 \\
\end{bmatrix}
}
\end{array} 
& 
\begin{array}{c}
\text{yes}\\
{
\tiny
\begin{bmatrix}
0 & 1 & 1 \\
0 & 2 & 1 \\
0 & 0 & 2 \\
\end{bmatrix}
}
\end{array} 
& 
\begin{array}{c}
\text{yes}\\
{
\tiny
\begin{bmatrix}
0 & 0 & 0 \\
0 & 0 & 0 \\
0 & 0 & 2 \\
\end{bmatrix}
}
\end{array} 
& 
\begin{array}{c}
\text{yes}\\
{
\tiny
\begin{bmatrix}
0 & 0 & 0 \\
0 & 0 & 0 \\
0 & 0 & 0 \\
\end{bmatrix}
}
\end{array} 
 \\[20pt]

\hline
(0, 0, 0, 12, 13 - 24, 14 + 23) 
& 
\begin{array}{c}
\text{generic}\\
{
\tiny
\begin{bmatrix}
0 & 0 & 1 \\
0 & 2 & 0 \\
0 & 0 & 0 \\
\end{bmatrix}
}
\end{array} 
& 
\begin{array}{c}
\text{yes}\\
{
\tiny
\begin{bmatrix}
0 & 0 & 1 \\
0 & 0 & 0 \\
0 & 0 & 0 \\
\end{bmatrix}
}
\end{array} 
& 
\begin{array}{c}
\text{yes}\\
{
\tiny
\begin{bmatrix}
0 & 0 & 0 \\
0 & 0 & 0 \\
0 & 0 & \frac{1}{2} \\
\end{bmatrix}
}
\end{array} 
& 
\begin{array}{c}
\text{yes}\\
{
\tiny
\begin{bmatrix}
0 & 0 & 0 \\
0 & 0 & 0 \\
0 & 0 & 0 \\
\end{bmatrix}
}
\end{array} 
 \\[20pt]

\hline
(0, 0, 0, 12, 14, 13 - 24) 
& 
\begin{array}{c}
\text{generic}\\
{
\tiny
\begin{bmatrix}
0 & 0 & 1 \\
0 & 0 & 0 \\
0 & 0 & 0 \\
\end{bmatrix}
}
\end{array} 
& 
\begin{array}{c}
\text{yes}\\
{
\tiny
\begin{bmatrix}
0 & 0 & 0 \\
0 & 0 & 1 \\
0 & 0 & 0 \\
\end{bmatrix}
}
\end{array} 
& 
\begin{array}{c}
\text{yes}\\
{
\tiny
\begin{bmatrix}
0 & 0 & 0 \\
0 & 0 & 0 \\
0 & 0 & 0 \\
\end{bmatrix}
}
\end{array} 
& 
\begin{array}{c}
\text{yes}\\
{
\tiny
\begin{bmatrix}
0 & 0 & 0 \\
0 & 0 & 0 \\
0 & 0 & 3 \\
\end{bmatrix}
}
\end{array} 
 \\[20pt]

\hline
(0, 0, 0, 12, 13 + 14, 24) 
& 
\begin{array}{c}
\text{generic}\\
{
\tiny
\begin{bmatrix}
0 & 0 & 1 \\
0 & 0 & 0 \\
0 & 0 & 0 \\
\end{bmatrix}
}
\end{array} 
& 
\begin{array}{c}
\text{yes}\\
{
\tiny
\begin{bmatrix}
0 & 0 & 0 \\
0 & 0 & 1 \\
0 & 0 & 0 \\
\end{bmatrix}
}
\end{array} 
& 
\begin{array}{c}
\text{yes}\\
{
\tiny
\begin{bmatrix}
0 & 0 & 0 \\
0 & 0 & 0 \\
0 & 0 & 0 \\
\end{bmatrix}
}
\end{array} 
& 
\begin{array}{c}
\text{yes}\\
{
\tiny
\begin{bmatrix}
0 & 0 & 0 \\
0 & 0 & 0 \\
0 & 0 & \frac{-1+2i}{5} \\
\end{bmatrix}
}
\end{array} 
 \\[20pt]

\hline
(0, 0, 0, 12, 13, 14 + 23) 
& 
\begin{array}{c}
\text{generic}\\
{
\tiny
\begin{bmatrix}
0 & 0 & 1 \\
0 & 0 & 0 \\
0 & 0 & 0 \\
\end{bmatrix}
}
\end{array} 
& 
\begin{array}{c}
\text{yes}\\
{
\tiny
\begin{bmatrix}
0 & 0 & 0 \\
0 & 0 & 1 \\
0 & 0 & 0 \\
\end{bmatrix}
}
\end{array} 
& 
\begin{array}{c}
\text{yes}\\
{
\tiny
\begin{bmatrix}
0 & 0 & 0 \\
0 & 0 & 0 \\
0 & 0 & 0 \\
\end{bmatrix}
}
\end{array} 
& 
\begin{array}{c}
\text{yes}\\
{
\tiny
\begin{bmatrix}
0 & 0 & 0 \\
0 & 0 & 0 \\
0 & 0 & -\frac{1}{3} \\
\end{bmatrix}
}
\end{array} 
 \\[20pt]

\hline
(0, 0, 0, 12, 13, 24) 
& 
\begin{array}{c}
\text{generic}\\
{
\tiny
\begin{bmatrix}
0 & 0 & 1 \\
0 & 0 & 0 \\
0 & 0 & 0 \\
\end{bmatrix}
}
\end{array} 
& 
\begin{array}{c}
\text{yes}\\
{
\tiny
\begin{bmatrix}
0 & 0 & 0 \\
0 & 0 & 1 \\
0 & 0 & 0 \\
\end{bmatrix}
}
\end{array} 
& 
\begin{array}{c}
\text{yes}\\
{
\tiny
\begin{bmatrix}
0 & 0 & 0 \\
0 & 0 & 0 \\
0 & 0 & 0 \\
\end{bmatrix}
}
\end{array} 
& 
\begin{array}{c}
\text{yes}\\
{
\tiny
\begin{bmatrix}
0 & 0 & 0 \\
0 & 2i & 0 \\
0 & 0 & -3i \\
\end{bmatrix}
}
\end{array} 
 \\[20pt]

\hline
(0, 0, 0, 12, 13, 14) 
& 
\text{no} 
& 
\begin{array}{c}
\text{generic}\\
{
\tiny
\begin{bmatrix}
0 & 0 & 0 \\
0 & 0 & 1 \\
0 & 0 & 0 \\
\end{bmatrix}
}
\end{array} 
& 
\begin{array}{c}
\text{yes}\\
{
\tiny
\begin{bmatrix}
0 & 0 & 0 \\
0 & 0 & 0 \\
0 & 0 & 2 \\
\end{bmatrix}
}
\end{array} 
& 
\begin{array}{c}
\text{yes}\\
{
\tiny
\begin{bmatrix}
0 & 0 & 0 \\
0 & 0 & 0 \\
0 & 0 & 0 \\
\end{bmatrix}
}
\end{array} 
 \\[20pt]

\hline
(0, 0, 0, 12, 13, 23) 
& 
\begin{array}{c}
\text{generic}\\
{
\tiny
\begin{bmatrix}
0 & 0 & 1 \\
0 & 0 & 0 \\
0 & 0 & 0 \\
\end{bmatrix}
}
\end{array} 
& 
\text{no}
& 
\begin{array}{c}
\text{yes}\\
{
\tiny
\begin{bmatrix}
0 & 0 & 0 \\
0 & 0 & 0 \\
0 & 0 & 2 \\
\end{bmatrix}
}
\end{array} 
& 
\begin{array}{c}
\text{yes}\\
{
\tiny
\begin{bmatrix}
0 & 0 & 0 \\
0 & 0 & 0 \\
0 & 0 & 0 \\
\end{bmatrix}
}
\end{array} 
\\[20pt]

\hline

(0, 0, 0, 0, 12, 15 + 34) 
& 
\text{no} 
& 
\begin{array}{c}
\text{generic}\\
{
\tiny
\begin{bmatrix}
0 & 0 & 1 \\
0 & 0 & 0 \\
0 & 1 & 0 \\
\end{bmatrix}
}
\end{array} 
 & 
\begin{array}{c}
\text{yes}\\
{
\tiny
\begin{bmatrix}
0 & 0 & 0 \\
0 & 0 & 0 \\
0 & 0 & 0 \\
\end{bmatrix}
}
\end{array} 
 & 
\text{no} \\[20pt]

\hline

(0, 0, 0, 0, 12, 15) 
& 
\text{no} 
& 
\begin{array}{c}
\text{generic}\\
{
\tiny
\begin{bmatrix}
0 & 0 & 1 \\
0 & 0 & 0 \\
0 & 1 & 0 \\
\end{bmatrix}
}
\end{array} 
 & 
\begin{array}{c}
\text{yes}\\
{
\tiny
\begin{bmatrix}
0 & 0 & 0 \\
0 & 0 & 0 \\
0 & 0 & 0 \\
\end{bmatrix}
}
\end{array} 
 & 
\text{no} \\[20pt]

\hline

(0, 0, 0, 0, 12, 14 + 25) 
& 
\text{no} 
& 
\begin{array}{c}
\text{generic}\\
{
\tiny
\begin{bmatrix}
0 & 0 & 1 \\
0 & 0 & 0 \\
0 & 0 & 0 \\
\end{bmatrix}
}
\end{array} 
 & 
\begin{array}{c}
\text{yes}\\
{
\tiny
\begin{bmatrix}
0 & 0 & 0 \\
0 & 0 & 0 \\
0 & 0 & 0 \\
\end{bmatrix}
}
\end{array} 
 & 
\begin{array}{c}
  \text{yes: see \cref{sec:rank0}}
\end{array}
\\[20pt]

\hline

(0, 0, 0, 0, 13 - 24, 14 + 23) 
& 
\text{no} 
& 
\begin{array}{c}
\text{generic}\\
{
\tiny
\begin{bmatrix}
0 & 0 & 1 \\
0 & 2 & 0 \\
0 & 0 & 0 \\
\end{bmatrix}
}
\end{array} 
 & 
\begin{array}{c}
\text{yes}\\
{
\tiny
\begin{bmatrix}
0 & 0 & 0 \\
0 & 0 & 1 \\
0 & 0 & 0 \\
\end{bmatrix}
}
\end{array} 
&
\begin{array}{c}
\text{yes}\\
{
\tiny
\begin{bmatrix}
0 & 0 & 0 \\
0 & 0 & 0 \\
0 & 0 & 0 \\
\end{bmatrix}
}
\end{array} 
  \\[20pt]

\hline

(0, 0, 0, 0, 12, 14 + 23) 
& 
\text{no} 
& 
\begin{array}{c}
\text{generic}\\
{
\tiny
\begin{bmatrix}
0 & 0 & 1 \\
0 & 0 & 0 \\
0 & 0 & 0 \\
\end{bmatrix}
}
\end{array} 
 & 
\begin{array}{c}
\text{yes}\\
{
\tiny
\begin{bmatrix}
0 & 0 & 0 \\
0 & 0 & 0 \\
0 & 0 & 0 \\
\end{bmatrix}
}
\end{array} 
 &
\begin{array}{c}
\text{yes}\\
{
\tiny
\begin{bmatrix}
0 & -i & 0 \\
0 & 0 & 0 \\
0 & 0 & 0 \\
\end{bmatrix}
}
\end{array} 
  \\[20pt]

\hline

(0, 0, 0, 0, 12, 34) 
&
\text{no} 
& 
\begin{array}{c}
\text{generic}\\
{
\tiny
\begin{bmatrix}
0 & 0 & 1 \\
0 & 0 & 1 \\
0 & 0 & 0 \\
\end{bmatrix}
}
\end{array} 
 & 
\begin{array}{c}
\text{yes}\\
{
\tiny
\begin{bmatrix}
0 & 0 & 0 \\
0 & 0 & 1 \\
0 & 0 & 0 \\
\end{bmatrix}
}
\end{array} 
 & 
\begin{array}{c}
\text{yes}\\
{
\tiny
\begin{bmatrix}
0 & 0 & 0 \\
0 & 0 & 0 \\
0 & 0 & 0 \\
\end{bmatrix}
}
\end{array} 
  \\[20pt]

\hline

(0, 0, 0, 0, 12, 13) 
& 
\text{no} 
& 
\begin{array}{c}
\text{generic}\\
{
\tiny
\begin{bmatrix}
0 & 0 & 1 \\
0 & 0 & 0 \\
0 & 0 & 0 \\
\end{bmatrix}
}
\end{array} 
 &
\begin{array}{c}
\text{yes}\\
{
\tiny
\begin{bmatrix}
0 & 0 & 0 \\
0 & 0 & 0 \\
0 & 0 & 0 \\
\end{bmatrix}
}
\end{array} 
 &
\begin{array}{c}
\text{yes}\\
{
\tiny
\begin{bmatrix}
0 & -1/6 & 0 \\
0 & 2 & 0 \\
0 & 0 & 2 \\
\end{bmatrix}
}
\end{array} 
  \\[20pt]
\hline
(0, 0, 0, 0, 0, 12 + 34) 
& 
\text{no} 
& 
\text{no} 
& 
\begin{array}{c}
\text{generic}\\
{
\tiny
\begin{bmatrix}
0 & 0 & 0 \\
0 & 0 & 1 \\
0 & 0 & 0 \\
\end{bmatrix}
}
\end{array} 
 & 
\begin{array}{c}
\text{yes}\\
{
\tiny
\begin{bmatrix}
0 & 0 & 0 \\
0 & 0 & 0 \\
0 & 0 & 0 \\
\end{bmatrix}
}
\end{array} 
  \\[20pt]
  
\hline

(0, 0, 0, 0, 0, 12) 
& 
\text{no} 
& 
\text{no} 
& 
\begin{array}{c}
\text{generic}\\
{
\tiny
\begin{bmatrix}
0 & 0 & 1 \\
0 & 0 & 0 \\
0 & 0 & 0 \\
\end{bmatrix}
}
\end{array} 
 & 
\begin{array}{c}
\text{yes}\\
{
\tiny
\begin{bmatrix}
0 & 0 & 0 \\
0 & 0 & 0 \\
0 & 0 & 0 \\
\end{bmatrix}
}
\end{array} 
  \\[20pt]

\hline

(0, 0, 0, 0, 0, 0) 
& 
\text{no} 
& 
\text{no} 
& 
\text{no} 
& 
\text{generic} \\
\hline
\end{longtable}
}

\medskip

{\small
\printbibliography
}

\end{document}